\documentclass[12pt]{amsart}
\usepackage{geometry}                % See geometry.pdf to learn the layout options. There are lots.
\usepackage{graphicx}
\usepackage{tikz-cd}

\usepackage{amsmath}
\usepackage{amsfonts}
\usepackage{enumerate}
\usepackage{amssymb}
\usepackage{epstopdf}

\newtheorem{theorem}{Theorem}[section]

\newtheorem{lemma}[theorem]{Lemma}
\newtheorem{corollary}[theorem]{Corollary}
\newtheorem{conjecture}[theorem]{Conjecture}

\theoremstyle{definition}
\newtheorem{definition}[theorem]{Definition}

\numberwithin{equation}{section}

\newtheorem{Thm}{Theorem}[section]
\newtheorem{Lem}[Thm]{Lemma}
\newtheorem{Pro}[Thm]{Proposition}
\newtheorem{Cor}[Thm]{Corollary}   
\newtheorem{Def}[Thm]{Definition} 

\newenvironment{Pf}[1]
{\trivlist\item[]{\it #1\@. }}{\hspace*{\fill}$\Box$\endtrivlist}

\newcommand{\R}{\mathbb{R}}
\newcommand{\Z}{\mathbb{Z}}

\renewcommand{\epsilon}{\varepsilon}

\renewcommand{\rho}{\varrho}

\long\def\forget#1\forgotten{} 
\begin{document}

%\section*{Introduction}

%\begin{Def}[Homological filling function.] \label{Fillfunc}

%\end{Def}

%In this paper we will proof the following theorems:

%\begin{Thm}  \label{TheoremA}

%\end{Thm}

%\begin{Thm} \label{TheoremB}

%\end{Thm}

%\begin{Thm} \label{TheoremC1}

%\end{Thm}

% ~\ref{figure1}

%\cite{GoHL}),

%%%%%%%%%%%%%%%%%%%%%%%%%%%%%%%%%%%%%%%%%%%%%%%%%%%%%%%%%%%%%%%%%%%%%%%%%%%%%%%

%%\section{The proof of Theorem ~\ref{TheoremA}}
 
\title{Geodesic nets on non-compact Riemannian manifolds}
                                             
\author[Chambers, Liokumovich, Nabutovsky, Rotman]{ Gregory R. Chambers, Yevgeny Liokumovich, Alexander Nabutovsky and Regina Rotman}
\date{\today}
\maketitle

\begin{abstract}
A geodesic flower is a finite collection of geodesic loops based at the same point $p$ that satisfy the following balancing condition: The sum of all unit tangent vectors to all geodesic arcs meeting at $p$ is equal to the zero vector. In particular, a geodesic
flower is a stationary geodesic net.

We prove that in every complete non-compact
manifold with locally convex ends
there exists a non-trivial geodesic flower.

\end{abstract}

\medskip\noindent
\section{Main results.}
\medskip

We begin with the following definition:
\begin{Def}
Let $M$ be a complete non-compact Riemannian manifold, and $\delta$ a positive number. Assume that there exists a finite collection of disjoint connected closed piecewise smooth hypersurfaces $\Sigma_i$, $i=1,2,\ldots$,
that divide $M$ into a disjoint union
of open submanifolds $M_0$, $M_1$, \dots, $M_e$ 
(for some integer $e$) 
so that:
\begin{enumerate}
    \item $M_0$ is bounded, and $M_i$ are unbounded for all $i>1$;
    \item If $i>0$, then the boundary of $M_i$ is $\Sigma_i$. The boundary of $M_0$ is the union of
all hypersurfaces $\Sigma_i$;
    \item (Locally convex ends condition) For each $i\geq 1$ there exists a positive $\epsilon$
%the closure $\bar M_i$ of $M_i$ is the convex subset of $M$. In
%other words,
such that each minimizing geodesic connecting every pair of $\delta$-close points of $\Sigma_i$ in $M$ is, in fact, in
$M_i \cup \Sigma_i$.
\end{enumerate}

Then we say that $M$ is a Riemannian manifold with $\delta$-locally convex ends, and
$M_0$ its core. If $M$ has $\delta$-locally convex ends for some $\delta>0$, we say that $M$ has locally convex ends.
\end{Def}

%Recall that a surface is piecewise smooth, if it
%is the image of a finite polyhedron under a continuous map to $M$ which is smooth on each open
%simplex.

Informally speaking, this condition means that each end of $M$ can be cut off by a smooth bounded
hypersurface that is locally convex to infinity.

We conjecture that:

\begin{conjecture} 
There exists a non-constant periodic geodesic on each complete non-compact Riemannian manifold 
with locally convex ends.
\end{conjecture}

\par\noindent
%{\bf Remark.} The finiteness of the set of separating hypersurfaces $\Sigma_i$ in the
%definition of manifolds with convex ends is not really necessary for this theorem to be true.
%The more general version can be proven along the same lines using Cech homology instead of
%singular homology.

To put this conjecture in proper perspective note that the well-known Fet-Lyusternik theorem asserts
that there exists a periodic geodesic on each closed Riemannian manifold. However, this is no longer
true in the non-compact case. The most obvious counterexamples are Euclidean
spaces (or, more generally, $\R^n$ endowed with any Riemannian metric of non-positive curvature.)
However, it is easy to exhibit examples with non-trivial topology. For example, one can
consider a surface of revolution in $\R^3$ obtained by rotating the hyperbola
$z={\frac{1}{x}}, x>0$ in the $XZ$-plane about the $Z$-axis. The fundamental group of this surface is $\Z$,
yet it is easy to see (using, for example, Clairaut's theorem) that there are no non-constant
periodic geodesics on this surface. It is easy to see why the standard existence proof developed
for the compact case does not work in this case: When we take a non-contractible closed
curve and start shrinking it, the curve slides to infinity, and does not converge to any
limiting closed curve.

Our main result asserts that the assumptions of this conjecture imply the existence of a closed geodesic net on the manifold, and, moreover of a geodesic flower.

\forget
%holds in the case when $M$ is the universal covering
%of a closed Riemannian manifold.

\begin{theorem} If $\tilde M^n$ is the Riemannian universal covering of a closed Riemannian
manifold $M^n$, and $\tilde M^n$ has locally convex ends, then there exists a non-trivial
contractible periodic geodesic on $M^n$.
\end{theorem}

One case, when $\tilde M^n$ has locally convex ends is that the domain $K^n$ exists already in $M^n$.

\begin{corollary} Let $M^n$ be a closed manifold, and $K^n$ a simply-connected open domain in $M^n$
such that its (not necessarily connected) boundary $\partial K^n$ is piecewise-smooth. Assume that for each pair of
sufficiently close in $M^n$ points $x,y$ of $\partial K^n$ the minimizing geodesic between them
is outside of $K^n$. Then there exists a non-trivial contractible periodic geodesic in $M^n$.
\end{corollary}

Indeed, one can lift $K^n$ to its isometric copy $\tilde K^n$ in $\tilde M^n$, and conisider
the boundary $\partial\tilde K^n$ that will be isometric to $\partial K^n$. As $x, y\in\partial\tilde
K^n$ will be in a small metric ball in $\tilde M^n$ isometric to its projection to $M^n$,
the minimizing geodesic between $x$ and $y$ will be outside of $\tilde K^n$ , and the previous theorem
implies the existence of a periodic geodesic on $M^n$. The corollary can be illustrated
by the following picture:

Note, that if $M^n$ is not aspherical, then the existence of the contractible periodic geodesic
can be derived via Morse theory on (non-contractible) space $\Lambda_{contr}M^n$ of contractible
closed curves on $M^n$ similarly to the proof of the classical Fet-Lyusternik theorem
asserting the existence of a non-trivial (not necessily contractible) periodic geodesic on
each closed Riemannian manifold. Yet if $M^n$ is aspherical, then $\Lambda_{contr}M^n$
is contractible, and the standard Morse-theoretic approach does not yield any critical points 
of the length functional. On the other hand, closed aspherical manifolds include not only
already quite interesting case of Riemannian tori $T^n$, but can have very diverse fundamental groups
(cf.[Davis]), and their universal covering spaces do not need to be homeomorphic to $\mathbb{R}^n$.
In fact,  aspherical manifolds can even have fundamental groups with unsolvable word problem - cf. [We],
Corolloary on p. 106. Yet any closed Riemannian manifold with such fundamental groups always have a contractible periodic geodesic - cf. [We], section 3.2.)

In the case of a general complete non-compact manifold we will prove here only the existence of geodesic
nets.
\forgotten

Recall, that geodesic nets on $M^n$ are critical points on the space of immersed (multi)graphs in $M^n$.
More formally, a geodesic net $N$ is a finite collection of points $v_i$ on $M^n$ (vertices of the
net) and (not necessarily different) smooth geodesics $\gamma_j$ starting and ending at vertices
$v_i$ of the net, where both ends of $\gamma_j$ can be the same vertex, so that for each
vertex $v_i$ the sum of the unit tangent vectors at $v_i$ to all geodesics $\gamma_j$ starting 
%or/and ending 
at $v_i$ is equal to the zero vector in $T_{v_i}M^n$.
Here we orient all the tangent vectors to $\gamma_j$ at $v_i$ from $v_i$ towards the other end
of $\gamma_j$. The geodesics $\gamma_i$ are called {\it edges} of a geodesic net. All periodic
geodesics are geodesic nets. (One can choose the corresponding multigraph as the graph with one vertex and one edge.)
Further, any union of a finite set of periodic geodesics is a geodesic net.
Yet a geodesic loop is a geodesic net if and only if it is a periodic geodesic. 
Similarly to periodic geodesics, geodesic nets are rare: for a generic 
%analytic 
Riemannian metric on a closed manifold the set of
geodesic nets is countable \cite{St}. Density and equidistribution
results for geodesic nets in Riemannian manifolds were proved recently in
\cite{LS} and \cite{LiS}.

Recently O. Chodosh and C. Mantoulidis showed \cite{CM}
that geodesic nets arising from Almgren-Pitts Min-Max theory
on surfaces are closed geodesics.
Very little is known about the existence of geodesic nets that are not unions of periodic geodesics. 
The only exception is a result of J. Hass and F. Morgan \cite{HM} asserting that for each convex Riemannian metric
on $S^2$ close to a round metric there exists a $\theta$-graph shaped geodesic net. (The underlying graph
consists of two vertices and three edges. In the geodesic net all edges must have non-zero length, and the stationarity condition
means that all angles between edges at each of two vertices are equal to ${2\pi\over 3}$.). On the other hand, it has been
proven in [NR] that the length of a shortest geodesic net on a closed Riemannian manifold $M^n$ does not exceed
$c(n)vol(M^n)^{1\over n}$, where $vol(M^n)$ denotes the volume of $M^n$, and $c(n)$ is an (explicit) constant that
depends only on the dimension. It also does not exceed $c(n) diam(M^n)$, where $diam(M^n)$ denotes the diameter of $M^n$.
Later it had been proven in [R2] (see also a later improvement in [R3]) that these results also hold for a special class of geodesic nets,  namely {\it geodesic flowers}.
A {\it geodesic flower} is a geodesic net that consists of a finite number of geodesic loops based at the same point $p$, and the sum
of all unit tangent vectors at $p$ (as usually directed from $p$) is equal to the zero vector (in the tangent space $T_pM^n$).
A geodesic flower is {\it non-trivial} if at least one of these geodesic loops is non-constant.

Here is our main theorem:

\begin{Thm} Let $M^n$ be a complete non-compact manifold with locally convex ends. Then there exists a
non-trivial geodesic flower on $M^n$.
\end{Thm}

The literature on the existence of periodic geodesics on complete non-compact Riemannian manifold is scarce.
We would like to note
papers by V. Benci and F. Giannoni (\cite{BG}) as well as
L. Asselle and M. Mazzuchelli (\cite{AM}) providing 
non-trivial sufficient conditions for the existence
of non-constant periodic geodesic on non-compact Riemannian manifolds.
These conditions prevent sliding to infinity of a cycle in the
space of closed curves in a compact ``core" of the manifold. In particular, they require the free loop space of the complete non-compact
Riemannian manifold to be non-contractible and, moreover, have non-trivial homology classes in a high dimension.
%(A simplest example of such a condiion,
%is the existence
%of hypersurfaces  that are convex to to the ``compact core" of the manifold that cut off all ends, and that the convex core has 
%a nontrivial homotopy group. In this case, the proof of the existence of a periodic geodesic will be essentially
%the same as in the compact case.)

In contrast, our result does not assume that complete non-compact Riemannian manifold has non-trivial topology
and are applicable to the case of manifolds diffeomorphic to $\mathbb{R}^n$. Further,
the locally convex ends condition does not prevent a sliding of cycles
in the space of free loops to infinity, but prevents movement of free loops
from infinity to the ``compact core" of the manifold.
%our surfaces cutting off the ends are convex to infinity (and concave to the
%``compact core" that is not required to have non-trivial topology). Thus, as we will see, in our case the hypersurfaces
%$\Sigma_i$ prevent transfer of non-trivial geometry and toplogy from the infinity to the compact core.
On the other hand , the main theorem in the
paper of V. Bangert, [B], asserting the existence of a periodic
geodesic on each complete surface of finite area is also applicable
to surfaces diffeomorphic to the plane. (An earlier paper
of G. Thorbergsson [T] established this result in the case of arbitrary surfaces with at least three ends,
so, in fact, [B] dealt with the case of surfaces with either two or one ends.
The first step in Bangert's proof was to observe that all ends of surfaces
of finite area can be cut off by appropriate geodesic loops with angles measured at the infinite side less than $\pi$.
Therefore, all surfaces of finite area have locally convex ends (in the sense of our definition). 
In particular, Bangert's work
also implies that the conjecture above is true for complete surfaces.)
The technique of Bangert is strictly two-dimensional.
Note that for any $n>2$ it is still not known
whether or not every complete $n$-dimensional Riemannian manifold of finite volume has a non-trivial periodic geodesic.
It is also interesting to notice that if $M$ is a complete surface that has a simple periodic geodesic
bounding a domain with compact closure in $M$, then the convex ends condition is trivially
satisfied: One can regard this geodesic as a one element set of hypersurfaces $\{\Sigma\}$. 

Below we will assume that $M^n$ is a complete non-compact Riemannian manifold with locally
convex ends. We will be looking for a periodic geodesic or, more generally, a geodesic flower %non-trivial periodic geodesic net or a periodic geodesic 
on $M^n$.

%The assumption that $M^n$ is a (possibly universal) covering of a closed Riemannian manifold will be used only at the very end of the
%proof of the existence of a periodic geodesic and will not be required for the proof of the existence of a geodesic net.
Our proof combines several main ideas.

%{\bf 2.1. Idea 1: Attempting an impossible filling}

\section{ Attempting an impossible filling.}

The first idea is yet another adaptation of Gromov's
technique from \cite{GrFill} based on attempting an impossible extension of the inclusion
of $M$ into a pseudo-manifold $W$ such that $\partial W=M$ to $W$. In our context this idea
works as follows. 

To explain the argument it will be convenient to define the following spaces
(the bar $\bar U$ over set $U$ denotes the closure of $U$):

\begin{itemize}
    \item $C_i$ is the cone over $\Sigma_i$, $i=1,\dots, e$;
    \item $S = M_0 \cup \bigcup_i C_i$;
    \item $CS$ is the cone over $S$;
    \item $M_i^c$ is the one point compactification of $\bar M_i$; 
    \item $M^c = M_0 \bigcup M_i^c$;
    \item $T=  M^c/(\bigcup_i M^c_i)$.
\end{itemize}

We have the following inclusion maps:
   $$ \phi: (\bar M_0,\bigcup_i \Sigma_i)\longrightarrow (M^c, \bigcup_i M^c_i) $$
   $$ j: (\bar M_0,\bigcup_i \Sigma_i)\longrightarrow (CS,\bigcup_i C_i)$$
   $$ \tau: \bar M_0/(\bigcup_i \Sigma_i)\longrightarrow T$$

Note that $\phi$ and $\tau$ induce isomorphisms on the $n$-th homology
of pairs with $\Z_2$ coefficients. Since $CS$ is contractible, $0=H_n(CS; \mathbb{Z}_2)=
H_n(CS, \bigcup_{i=1}^e p_i;\mathbb{Z}_2)=H_n(CS,\bigcup_{i=1}^e C_i;\mathbb{Z}_2)$, where $p_i$ denote tips of the cones $C_i$.

\forget

For each $i$ let $C_i$ denote the cone over $\Sigma_i$, and $S$ denote the union of all cones $C_i$ and $M_0$. 
We regard
% $P$ % YL: what is $P$?
$C_i$ % YL: changes $P$ to $C_i$
as a topological space not endowed with any metric.
%On the other hand all $\Sigma_i$ can be triangulated, and these
%triangulations extend to a smooth triangulation of $\bar M_0$. (Here and below $\bar X$ denotes the closure of $X$.)
%Moreover, they can be extended to a triangulation of $S$ , where the triangulation of $C_i$ is defined as the cone over the chosen
%triangtlation of $\Sigma_i$. 
Further, let $CS$ denote the cone over $S$, and $M_i^c$ the one point compactification of $\bar M_i$. (The bar over $M_i$ denotes here and below the closure of $M_i$.)
Finally, let $M^c$ denote the union of $M_0$ and all $M^c_i$, or, less formally, the compactification of $M$ with $e$ points,
$c_1,\ldots ,c_e$ by $e$ points at ``infinity", one for each ``end" $M_i$. 

%endowed 
%with a (singular) Riemannian metric extending the metric on $\Sigma_i$. Here is a concrete
%way to construct a metric on $C_i$: First, isometrically embed $\Sigma_i$ in a Euclidean space $\R^N$.
%Then assume that $\R^N$ is a hyperplane in $R^{n+1}$. Choose a point $c\in\R^{n+1}\setminus\R^n$ ,
%and define $C_i$ as the union of straight line segments connecting $c$ with all points of $\Sigma_i$
%in $\R^{n+1}$. Denote $M_0\bigcup_i C_i$ by $S$. So, $S$ includes the ``compact core" $M_0$ of $M$
%and cones that close off the ``holes" $\Sigma_i$.
Note that $H_n(S;\Z_2)= H_n(S, \bigcup_i C_i;\mathbb{Z}_2)=H_n(\bar M_0,\bigcup_i\Sigma_i; \mathbb{Z}_2),$ where $n$ denotes the dimension of $M$. 
Similarly, $H_n(M^c, \bigcup_{i=1}^e M^c_i;\mathbb{Z}_2)=H_n(M,\bigcup_{i=1}^e M_i;\mathbb{Z}_2)=
H_n(\bar M_0,\bigcup_{i=1}^e\Sigma_i;\mathbb{Z}_2)=\mathbb{Z}_2$. Let $\phi: (\bar M_0,\bigcup_i \Sigma_i)\longrightarrow (M^c, \bigcup M^c_i)$
be the inclusion map. It is clear that $\phi$ induces the isomorphisms
between $n$th homology groups of these pairs with $\mathbb{Z}_2$ coefficients, where both of these groups are isomorphic to $\mathbb{Z}_2$.
Note that $M_0$ can be regarded as a subset of $S$, and $S$ is the base of the cone $CS$. Denote
the corresponding inclusion of pairs $(\bar M_0,\bigcup_i \Sigma_i)\longrightarrow (CS,\bigcup_i C_i)$ by $j$.
Note, that since $CS$ is contractible, $0=H_n(CS; \mathbb{Z}_2)=
H_n(CS, \bigcup_{i=1}^e p_i;\mathbb{Z}_2)=H_n(CS,\bigcup_{i=1}^e C_i;\mathbb{Z}_2)$, where $p_i$ denote tips of the cones $C_i$.
Consider the quotient map $q$ from $M^c$ to $T$. The restriction
of this map to $M_0$ is a homeomorphism. Alternatively, one could have defined $T$ as a quotient space of $\bar M_0$ by identifying all points
of surfaces $\Sigma_i$  into one point $m$. As all homology groups of the pairs $(\bar M_0, \bigcup\Sigma_i)$ and
$(M^c,\bigcup_i M^c_i)$ are isomorphic to the $n$th homology groups of the corresponding quotient spaces, both of which are $T$,
one can compose $\phi_*$ with this isomorphism, and obtain the identity isomorphism $\tau_{*}: H_n(\bar M_0/(\bigcup_i\Sigma_i);\mathbb{Z}_2)=H_n(T;\mathbb{Z}_2)\longrightarrow H_n(M^c/\bigcup_i M^c_i;\mathbb{Z}_2)=H_n(T; \mathbb{Z}_2)$. Here $\tau$ is the homeomorphism
$\bar M_0/(\bigcup_i \Sigma_i)\longrightarrow M^c/(\bigcup_i M^c_i)$, which is equal to the identity map on $M_0$ regarded as a subset
of both the domain and the target space, and sending the point corresponding to $\bigcup_i\Sigma_i$ to the point corresponding to
$\bigcup_i M^c_i$.

Now, the following lemma is obvious:
\forgotten

The following lemma easily follows:

\begin{Lem} \label{no extension}
%Assume that $\phi_i:\C_i\longrightarrow M^c_i$, $i=1,\ldots ,e$, are maps extending
%the inclusion maps of $\Sigma_i$ regarded as the base of the cone $C_i$ into $M\subset M^c$.
%It is not required that $\phi_i$ is continuous, but it is assumed that $\phi_i$ is continuous 
Let $\phi': (\bar M_0,\bigcup_i \Sigma_i)\longrightarrow (M^c,\bigcup_i M^c_i)$
be homotopic to the inclusion $\phi$.
Then there exists no extension of
$\phi'$
%: (\bar M_0,\bigcup_i \Sigma_i)\longrightarrow (M^c,\bigcup_i M^c_i)$ 
to $(CS, \bigcup_i C_i)$.
%, where $\bigcup_i C_i$ in the last expression
%is a subset of the base of the cone.  
In other words,
there is no continuous map $\psi: (CS, \bigcup C_i)\longrightarrow (M^c, \bigcup_i M^c_i)$
such that $\phi'=\psi\circ j$. 

\begin{center}
\begin{tikzcd}
(\bar M_0,\bigcup_i \Sigma_i) \arrow[r, "\phi'"] \arrow[d, "j"] & (M^c, \bigcup_i M^c_i)
 \arrow[r, "q"] & T \\
(CS,\bigcup C_i) \arrow[ru, dashrightarrow, "\psi"] \arrow[rru, dashrightarrow, "\tilde{\psi}"]& &
\end{tikzcd}
\end{center}

Moreover, the last assertion remains true if the continuity of $\psi$ is replaced by a weaker condition that
the composition $\tilde\psi:CS\longrightarrow T$ of $\psi$ with the quotient map $q: (M^c, \bigcup_i M^c_i)\longrightarrow M^c/(\bigcup_i M^c_i)=T$
is continuous.
\end{Lem}

\begin{proof}
Indeed, if lemma is false, then $\phi'=\psi\circ j$, and, therefore, $q\circ\phi'=q\circ\psi\circ j=\tilde\psi\circ j$ is continuous, where
both maps are regarded as defined on $\bar M_0$.
As both maps send $\bigcup_i\Sigma_i$ to the point $m$ of $T$, they can be lifted to equal maps of $\bar M_0/(\bigcup_i\Sigma_i)$.
In other words, $q\circ\phi'=\tau^{-1}\circ q\vert_{\bar M_0}$, %for some $\lambda: \bar M_0/(\bigcup_i\Sigma_i)\longrightarrow M^c/(\bigcup_i M^c_i)$. It is easy to see that $\lambda=\tau$, andr
where, as we know, $\tau^{-1}$ 
induces the identity isomorphism of the $n$th homology groups with $\mathbb{Z}_2$
coefficients. Similarly, $\tilde\psi\circ j$ can be represented as $\mu\circ q\vert_{\bar M_0}$ for some $\mu$, which must coincide
with $\tau^{-1}$. But note that
$\mu$ can be decomposed as $\tilde j: \bar M_0/(\bigcup_i\Sigma_i)\longrightarrow CS/(\bigcup_i C_i)$, and the map
$\rho: CS/(\bigcup C_i)\longrightarrow T=M^c/(\bigcup_i M^c_i)$ induced by $\tilde\psi$.
As the $n$th homology with $\mathbb{Z}_2$ coefficients
of $CS/(\bigcup_i C_i)$ is trivial, the induced homomorphism $\tilde j_*$ is zero, and, therefore, $\mu_*$ is also trivial. But $\mu=\tau^{-1}$,
and we obtain a contradiction proving the lemma.
\end{proof}

\forget
We will also need space $W$, constructed as follows.

Now construct a space $W$ as follows. Consider cones $C\bar M_0$ with vertex $P_0$ over $\bar M_0$,  and $CC_i$ over all cones
$C_i$, $i=1,\ldots, e$, with vertices $P_i$.
If we would glue $C\bar M_0$ to $CC_i$ for all $i=1,2,\ldots, e$ identifying $C\Sigma_i$ in boundaries of both cones, we are going
to obtain $CS$. Instead, for each  $i$ we only identify $\Sigma_i$ in the boundary of both cones, and then fill the suspension
$S\Sigma_i$ of $\Sigma_i$ obtained by gluing the cones $C\Sigma_i$ with vertices $P_0$ and $P_i$ by the cone $CS\Sigma_i$ over this suspension
with the new cone vertex $Q_i$. Denote the resulting space by $W$. Observe, that $W$
is glued from the cones $C\bar M_0$, $CC_i$ and $CS\Sigma_i$. Note that the contractible set $CS\Sigma_i$ projects to the cone 
$C\Sigma_i$ with the tip at $Q_i$ by the map that sends each point $tQ_i+(1-t)A$, where $A=\tilde t P_i + (1-\tilde t)C$ or
$A=\tilde t P_0+(1-\tilde t)C$ for $C\in \Sigma_i$, to $t\tilde (1-(1-t)(1-\tilde t)) Q_i+(1-t)(1-\tilde t)C$. Thus, we obtain a surjective map
$\theta:W\longrightarrow CS$. We have an obvious inclusions of $C_i$ as bases of the cones $CC_i\subset W$, and $\bar M_0$ as the base
of the cone $C\bar M_0\subset W$, and, thus, an inclusion $J: (\bar M_0,\bigcup_i \Sigma_i)\longrightarrow (W,\bigcup_i C_i)$. Note that $\theta\circ J=j$.

Now note that $\theta$ is the identity on $S$ that is regarded as a subset of $W$. Therefore, we can define a partial inverse $\theta^{-1}$ on the base $S$ of $CS$.
It will map $S\subset CS$ to $S$ in the base of $W$. In particular, $\Sigma_i\subset S$ will be mapped to $\Sigma_i\subset S\Sigma_i\subset W$ , where the first inclusion send $\Sigma_i$ to the base of the suspension.
This partial inverse of $\theta$ can be easily extended to a map $\alpha:CS\longrightarrow W$ sending $C\Sigma_i\subset CS$ to $C\Sigma_i\subset CS\Sigma_i\subset W$, where the first of the two inclusions sends $\Sigma_i$ to the base of the suspension. Obviously, $\alpha\circ j=J$.

The previous lemma has the following immediate corollary:

\begin{Cor}
There is no (not necessarily continuous) map $\Psi: (W, \bigcup C_i)\longrightarrow (M^c, \bigcup_i M^c_i)$
such that 1) $\phi=\Psi\circ J$; 2)
the composition $\tilde\psi:W\longrightarrow T$ of $\Psi$ with the quotient map $q: (M^c, \bigcup_i M^c_i)\longrightarrow M^c/(\bigcup_i M^c_i)=T$
is continuous.
\end{Cor}

Indeed, assume that such a map $\Psi$ exists. Precomposing it with $\alpha:CS\longrightarrow W$ we obtain a map $\psi$ satisfying the conditions of the previous lemma that asserts that such a $\psi$ cannot exist.
%and restriction to the base $\bar M_0$ of $C\bar M_0\subset CS$ can be composed with $\theta$, and $\theta\circ\Psi$ would provide a counterexample to the previous lemma.

The space $W$ will play a prominent role in this paper, and we would like to endow it with a canonical triangulation
extending any (fine) triangulation of the pair $(\bar M_0,\bigcup_i\Sigma_i)$. Given such a triangulation we first
triangulate each $C_i$ using the cone triangulation over given triangulation of $\Sigma_i$ with
a new vertex $q_i$ (the tip of the cone), then triangulate $CC_i$ and $C\bar M_i$ using the cone triangulation
with verices $P_i$, $i\in\{0,\ldots ,e\}$. Finally we triangulate $CS\Sigma_i$ using the cone triangulation over the already
constraucted triangulation of the boundary with the new vertex $Q_i$ for each $i=1,\ldots ,e$.
%Let us introduce another notation: Denote a metric cone over $S$ by $CS$. This cone can be constructed
%as $C_i$: One can start with an isometric embedding of $M_0$ in a Euclidean space of dimension $M$,
%embed this Euclidean space in $\R^{M+e+1}$, where $e$ is the number of hypersurfaces $\Sigma_i$,
%cone off all $\Sigma_i$ so as to obtain $C_i$ using a new extra dimension each time, and finally
%cone-off $S$ using a point ``above" the previously used hyperplane $\R^{M+e}$.
%\par
%Finally, we conider one point compactifications at infinity $M^c_i$ of
%$\bar M_i=M_i\bigcup\Sigma_i$ as well as the compactification $M^c$ of $M$ obtained by adding a point
%at infinity for each $M_i$. In other words, $M^c=M_0\bigcup_i M^c_i$. Note, that $M^c$ and $M^c_i$
%are regarded purely as topological objects (unlike, $M, M_i, S_i, S, CS$).

Now we are going to explain how to prove our main theorem asserting the existence of a geodesic net by contradiction. 
Assume that $M$ does not have a geodesic net. Our strategy will be to bring this assumption to a contradiction
with the previous corollary. We will use the assumption to extend $\phi$ from $\bar M_0$ (regarded as a subset of the base of $W$) to a map $\Psi$ defined on the whole $W$, so that the restriction
of $\Psi$ to $C_i$ takes values in $M^c_i$, and the composition of $\Psi$ with the quotient map to $T=M^c/(\bigcup M^c_i)$
is continuous. The previous corollary would provide the desired contradiction.
\forgotten

Now we will sketch the plan of the proof of our main theorem by contradiction. We are going to assume that there are no geodesic flowers on $M$, and will try to extend the identity map $\phi$ to a map $\psi$ of $CS$ that cannot exist according to the previous lemma.
%We call a triangulation of a subset of $M_0$ {\it fine}, if for each simplex $\Delta$ we
%have $\Delta \subset B_{r_c(x)}(x)$ for every point $x \in \Delta$, where $r_c$
%is the convexity radius of $M_0$ at $x$.
Choose a very fine triangulation of $\bigcup_i\Sigma_i$, extend it to a very fine triangulation of $\bar M_0$, and also extend it to a (not fine) triangulation of all $C_i$, $i\geq 1$ by adding new vertices $q_i$ (= tips of cones $C_i$) and triangulating all $C_i$ as cones over the chosen fine triangulations of bases of $C_i$. Finally, choose a vertex $P_0$ of the cone $CS$, and triangulate $CS$ as the cone over just constructed triangulation of the base.

We are going to attempt the impossible extension of $\phi$ to $\psi$. We map $P_0$ to any vertex $p$ in the fine triangulation of $\bar M_0$ in $M_0$,
$q_i$ to any vertex $v_i$ of the fine triangulation of $\Sigma_i$. Map all $1$-dimensional simplices of $S$ that connect $q_i$ with a vertex
$v$ of $\Sigma_i$ to a minimizing geodesic on $\Sigma_i$ in its intrinsic metric that connects $v_i$ and $v$. Observe that $1$-skeleta of
all $n$-dimensional simplices of $S$ with a vertex at $q_i$ will be on $\Sigma_i$. We are going to map all ``new" $1$-simplices of $CS$ with one vertex at $P_0$ to a minimal geodesic in the intrinsic metric of $\bar M_0$ that connects $p$ and this vertex (or $p$ and $v_i$, if the vertex is $q_i$).
Observe that $1$-skeleta of all $(n+1)$-dimensional simplices of $CS$ will be in $\bar M_0$.

\forget
We also extend the chosen triangulation of $M_0$ to a triangulation
of $M$. We start from extending $\phi$ to new $0$- and $1$-dimensional simplices of
$W$. The vertex $q_i$ of the cone $C_i$ can be mapped to any vertex $\tilde q_i$ of the triangulation of $M_i$. The vertex $P_0$ will be mapped to any fixed vertex $v_0$ of the triangulation of $M_0$, $P_i$ will be mapped to any vertex $\tilde Q_i$.
of the triangulation of $M_i$ for $i=1,\dots ,e$. We map each vertex $Q_i$ to some vertex of the chosen triangulation
of $\Sigma_i$.

Now we need to map all new $1$-dimensional simplices of $W$ to curves connecting the images of their endpoints.
The edges of $C_i$ will connect points in $\bar M_i$, ($i\geq 1$), and can be mapped as (arbitrary) minimizing geodesics
in the inner metric of $\bar M_i$. Similarly, all new $1$-dimensional simplices in $C\bar M_0$ must be mapped into curves between two points
of $\bar M_0$, and are chosen as minimizing geodesics in the inner metric of $M_0$. The new edges of $CC_i$ are mapped into minimizing
geodesics in $\bar M_i$ for $i=1,\ldots, e$. Finally, the new edges in each $CS\Sigma_i$ are mapped into minimizing geodesics in 
the inner metric of $\Sigma_i$.

Observe, that our choice of the triangulation of $W$ and the extension of $\Psi$ to the $1$-skeleton implies that the $1$-skeleton
of each $(n+1)$-dimensional simplex of $W$ (that is, a simplex of maximal dimension) will be mapped either to $\bar M_0$, or
to $\bar M_i$ for some $i\geq 1$. (These two possibilities are not mutually exclusive, as
the $1$-skeleta of some simplices of the maximal dimension are mapped to $\Sigma_i$.) Thus:

\begin{Lem}
If the $1$-skeleton of some $(n+1)$-dimensional simplex of $W$ is not mapped to the complement of $M_0$, then it is mapped to the closure of $M_0$.
\end{Lem} 

\forgotten
\par
A natural idea is to continue extending $\psi$ inductively to $2$-skeleta, then to $3$-skeleta, etc. Instead, we assume that $M_0$
does not have any geodesic flowers and {\it immediately} extend the inclusion of the $1$-skeleton of each $(n+1)$-dimensional simplex of $CS$ into $M$
to the map of the whole simplex. We do this using the idea of
``filling of cages" from [R1]. This idea will be described in details in the next two sections.  The description of ``filling of cages" in
the next two sections will immediately imply that the image of each $n$-simplex with vertex at $q_i$ in $S$ will be in $\bar M_i$, $i\geq 1$. 

Here we only note that we would like to be able to eventually extend
maps of the $1$-skeleton of the standard $(m+1)$-dimensional simplex $\Delta^{m+1}$  for each $m\in \{1,\ldots, n\}$ to its interior. The extension must map the interior
to $M^c$. If the $1$-skeleton is mapped to one of $\bar M_i$ (for example, to $\Sigma_i$), then the extension maps the whole simplex to
$M_i^c$.
The restriction of each of these extensions
to the $1$-skeleton of any $(l+1)$-dimensional face of $\Delta^{m+1}$ must be the extension of the restriction. We want these extensions only for maps , where the
length of the image of each side of $\Delta^{m+1}$ does not exceed $\tilde L=\max\{Diam(\bar M_0), \max_{j\geq 1} Diam(\partial M_j)\}$, where all the diameters are calculated with respect to the inner metrics. Finally, mention that very small cages must be filled by very small simplices.

Combining all the extensions for all $(n+1)$-dimensional simplices of the triangulation of $CS$, we will obtain
a map $\psi:CS\longrightarrow M^c$. We would like to argue that
Lemma 2.1 yields
a contradiction that refutes the assumption that there are no geodesic flowers on $M$.
Yet a minor technical difficulty is that the restriction of $\psi$ to $\bar M_0$ 
is not the identity map on $\Sigma_i\subset S$ as in the conditions of Lemma 2.1. (So, the composition of $j$ and $\psi$ is not the inclusion $\phi$.) This happens because the fillings of $1$-skeleta of very small simplices $\sigma^{n-1}$ of the triangulations of $\Sigma_i$ will be not $\sigma^{n-1}$ but very close very small $(n-1)$-dimensional simplices contained in a small collar of $\Sigma_i$ in $\bar M_i$.
%(We will explain the details in the next section.)
%In the next section we will see that this map will be very close to the inclusion map.
For each $i$ there will be an obvious homotopy between the restriction of $\psi$ to $\Sigma_i$ and the inclusion map of $\Sigma_i$ to $M$ that will
move all points in the fillings of $1$-skeleta
of small simplices triangulating $\Sigma_i$ via normals to $\Sigma_i$ to $\Sigma_i$ inside $M_i$. This homotopy can be extended to a homotopy between the restriction of $\psi'$ to $\bar M_0$ and the inclusion of $\bar M_0$ to $M$. 
%(The same is true for the restriction of this map to $\Sigma_i$ for each $i\geq 1$.)
Hence, applying Lemma \ref{no extension} with $\phi = \psi|_{\overline{M}_0}$
homotopic to the inclusion $\phi$
we obtain the desired contradiction.
\forget
This homotopy (homotopies)
will be relative to the $1$-skeleton. This homotopy will move all points along a continuous family of minimizing geodesics in $M$, and the image
of this homotopy will be contained in a small neighbourhood of $\bar M_0$ (respectively,  $\cup\Sigma_i$). (Here the idea is that the image of each small simplex of $\Sigma_i$ under ``filling of cages" map will be contained in the same small convex metric ball of $\Sigma_i$ as the original simplex.) 
\forgotten
%The domain of the homotopy between the restriction of ``filling cages" map to $\Sigma_i$ and the inclusion of $\Sigma_i$ in $\bar M_0$ 
%is the annulus $\Sigma_i\times [0,1]$. We can glue this domain to $C\Sigma_i$ obtaining a ``long" cone $C\Sigma_i$ of height $2$ that can be rescaled back to the standard
%$C\Sigma_i$ (of height $1$). The boundary of this ``net" $C\Sigma_i$ is mapped via the "filling of cages" map, a collar near the boundary
%is the homotopy with the image in a small neighbourhood of $\Sigma_i$ and the ``inner" cone $C\Sigma_i$ is mapped to $M_0$ exactly as described above. (This picture will be the same for all $i\geq 1$.)

\forget
Now the idea is that in order to construct $\psi$ exactly as in Lemma 2.1 one can extend this homotopy between two maps of $\bar M_0$ to a homotopy
between two maps of the ambient $S$ that will be defined below. This ambient homotopy will be defined
on $S\times [0,1]$. Attaching $S\times [0,1]$ to $CS$, we obtain a larger cone (of height $2$) that can be rescaled back to $CS$. The restriction of the map $\psi$ of the larger cone to $\bar M_0$ in its
base is the inclusion map, the restriction of this map
to the collar (= $\bar M_0\times [0,1]$ in the glued $S\times [0,1]$) is the homotopy between the inclusion map and the ``filling of cages" map, the restriction to the``inner cone" (=the ``initial" $CS$) is $\psi'$, that is, the gluing of all filling of cages maps for the $1$-skeleta of all  $(n+1)$-dimensional simplices of $CS$. It remains
only to define $\psi$ on $S\times [0,1]$.

But first we need to define the map of the base of the ``longer" $CS$ on the complement of $\bar M_0$. (Recall, that this map was just defined on $\bar M_0$ as the inclusion.) This complement consists of cones $C_i=C\Sigma_i$ for all $i$. We can think of each $C\Sigma_i$ as the union
of a collar $\Sigma_i\times [0,1]$ and an inner cone $C\Sigma_i$. The inner cone is mapped using the restriction of $\psi'$ to $C_i$. In particular,
the restriction of this map to the boundary of the inner $C\Sigma_i$ is the map
of $\Sigma_i$ obtained from the $1$-skeleton of the fine triangulation of $\Sigma_i$
via the filling of cages. The map of the collar is a homotopy $H: \Sigma_i\times [0,1]\longrightarrow M$ between the inclusion (at $t=0$) and the ``filling of cages" map of $\Sigma_i$ to $M$ (at $t=1$).

We already know that the restriction of the map $\psi$ to the collar $S\times [0,1]$ to
$\bar M_0\times [0,1]$ must be a homotopy between the inclusion (at $t=0$)
and the ``filling of cages" map (at $t=1$). We define restriction of this homotopy to $\Sigma_i$ in the boundary of $M_0\subset S\times \{0\}$ as the same homotopy $H:\Sigma_i\times [0,1]\longrightarrow M$ as above.
The restriction of this homotopy to the inner cone of $C$ is the identity map at each moment of time. However, we make the collar between the inner cone in $C\Sigma_i$ and the outer boundary
of $C\Sigma_i$ to become thinner and thinner as $t$ increases. At each moment its width is $1-t$. At the moment $t=1$ its
width becomes zero. At each moment $t$, the collar $\Sigma_i\times [t,1]$ of $\sigma_i$ in $C\times \{t\}$
is mapped using the restriction of $H$ to $\Sigma_i\times [t,1]$.
This completes the construction of $\psi$.

%The reader will find the missing details in this description at the end of the next section.

We are going to consider this homotopy in reverse , that is from the identity map to
the constructed map of $\Sigma_i$. We are going to attach its domain , $\Sigma_i\times [0,1]$ to $C_i=C\Sigma_i$ obtaining a longer cone over $\Sigma_i$ that can be rescaled back to $C_i$. The base $\Sigma_i$ of this ``longer" cone is mapped to $M^c$ via
the identity map, the collar of the base is mapped using the mentioned homotopy,
and the rest is mapped using the "filling of cages". Note that while this "filling
of cages" is the usual continuous filling of the $1$-skeleta of small simplices in
$\Sigma_i$, or, more generally, $\bar M_0$ is a usual continuous filling,
it will not be continuous on ``larger" simplices with a vertex at $q_i$ (or at $P_0$ below). However, these discontinuities will disappear after composing with the quotient map that identifies the complement of $M_0$ to a point.

%$(\bar M_0,\cup_i \Sigma_)$ will be homotoped to the identity map. (The homotopy will simply connect the respective points along the unique minimizing geodesics.) Gluing the homotopy between the identity map and the constructed map of the base of $CS$
%at the bottom of the constructed map of the cone $CS$, we obtain a map $\psi$ with the desired properties. (The domain of the homotopy can be regarded as $S\times [0,1]$. When one  glues it to the base of $CS$, one obtains another cone over $S$ with height $2$. Then one can rescale the height of the ``new" $CS$ back to $1$.)  This contradiction will refute our assumption that no-trivial geodesic flowers on $CS$ do not exist.

The missing details of this sketch of the proof will be explained 
and will become clear at the end of the next section.
%its $1$-skeleton.
\forgotten

\medskip\noindent
\section{Cages and their fillings.}

\subsection{Cages, flowers, fillings.}
The content of this section follows [R1], adapting one of the ideas to our
situation. 
For each $i=2,3,\ldots, n+1$ define an {\it $i$-cage} in $M$ as a map of the $1$-skeleton of the standard $i$-dimensional
simplex $\sigma^i$ to $M$, where each edge is mapped into a broken geodesic. If this map is constant, we call the cage {\it a constant cage}.
We extend the class of all constant cages, and, thus, all cages by allowing constant cages that map the $1$-skeleton of $\sigma^i$ to
one of $e$ points at infinity of $M^c$. Thus, formally speaking, cages map the $1$-skeleta of simplices to $M^c$, but the image of each non-constant cage
is required to be in $M$. We are assuming that the vertices of $\sigma^i$ form a totally ordered set (the ordering might come from
a numbering of the vertices by numbers $1,\ldots ,N+1$.) Denote
the space of all $i$-cages in $M$ by $Cage_i$. Note that for each $j$-dimensional face of $\sigma^i$ the restriction
of an $i$-cage $k$ to the $1$-skeleton of $\sigma^i$ is a $j$-cage. We will call such $j$-cages {\it subcages} or {\it $j$-subcages}
of $k$.

We will also need a special class of $i$-cages that we will call {\it flowers}. Consider the vertex $v$ of an $i$-cage with the highest number.
Assume that {\it all} edges (=$1$-simplices) incident to $v$ are mapped to the same point. Thus, such cages can be also defined as maps of  
${i(i-1)\over 2}$ circles to $M$. In fact, some of these circles can also be mapped to a point. We use notation $Flwr_k$ 
for maps of the wedge of $k$ loops to $M$. As we just observed, there are natural inclusions  $Flwr_{i(i-1)\over 2}\subset Cage _i$, and $Flwr_j\subset Flwr_i$ for $j\leq i$. Recall, that a set $C \subset M$ is $\delta$-locally convex, if for all $x,y\in C$ that are $\delta$-close in $M$ all minimizing
geodesics in $M$ between $x$ and $y$ are contained in $C$.
For each $L \in (0, \infty)$ let $Cage_i^L$ denote the subset of $Cage_i$ formed
by all cages of length $\leq L$, $Cage^\infty_i$, by definition, coincides with $Cage_i$. Similarly, $Flwr^\infty_j=Flwr_j$, and for each $L\in (0,\infty)$ $Flwr_j^L$ denotes a subset of $Flwr_j$ formed by all $j$-flowers of length $\leq L$.
\begin{Def}
Given $\delta>0$, and $L\in (0,\infty]$ a {\it strong filling of $(n+1)$-cages} of length $\leq L$ is a family of continuous maps $H_i:Cage^L_i\times [0,1]\longrightarrow Cage_i$ for $i\in\{2,\ldots,n+1\}$ such that:
\begin{enumerate}
    \item For each $i$-cage $k$ $H_i(k,0)=k$ and $H_i(k,1)$ is a constant cage. In other words, $H_i$ is a contraction of $Cage^L_i$ inside $Cage_i$ to its subspace formed
by all constant cages.
    \item If the image of an $i$-cage $k$ is in a $\delta$-locally convex subset $V$ of $M$ (in particular, $k$ might be in the closure of $M_j$ for some $j$), then its trajectory $H(k,t),\ t\in [0,1)$ will be in $V$.
    \item If $k$ is a constant $i$-cage, then $H_i(k,t)=k$ for all $t \in [0,1]$.
    \item If $k\in Flwr^L_j\subset Flwr^L_{i(i-1)\over 2}\subset Cage^L_i$, $(j\leq {i(i-1)\over 2})$, then for $H_i(k,t)\in Flwr_j$ for all $t \in [0,1]$.
\end{enumerate}
\end{Def}

The importance of this notion lies in the following proposition:

\begin{Pro}
Assume that $M$ admits a strong filling of cages for some $\delta>0$ and $L\in (0,\infty]$.
Let $k$ be an $i$-cage in $M$ of length $\leq L$, $L<\infty$, for some $i=2,3,\ldots ,n+1$. Then:
\begin{enumerate}
    \item There exists a continuous map $\phi=\phi(k)$ of the $i$-simplex $\sigma^i$ to $M^c$ extending $k$.
Further, the dependence of $\phi$ on $k$ is continuous
(that is, the restriction of $\phi(k)$ to the $1$-skeleton 
of $\sigma^i$ is equal to $k$). 
%More formally, $\phi(k)$ is a restriction of a continuous map $\Phi:Cage^L_i\times \sigma^i\longrightarrow M^c$ to
%the points, where the value of the first variable is $k$.
    \item Let $\sigma^j$ be a $j$-dimensional face of $\sigma^i$, $2\leq j<i$, $k_j$ the corresponding $j$-subcage of $\sigma^i$. Then $\phi(k_j)$ 
coincides with the restriction of $\phi(k)$ to $\sigma^j$.
    \item If the
image of $k$ (or one of its $j$-dimensional subcages $k_j$) is in a closed $\delta$-locally convex subset of $M$ (in particular, it might be in the closure of $M_m$ for some $m=1,\ldots ,e$), the same will be true for the image
of $\phi$ (correspondingly, the restriction of $\phi$ to the $j$-dimensional face $\sigma^j$ of $\sigma^i$ corresponding to $k_j$.)  
\end{enumerate}
\end{Pro}

\begin{proof}
The proof uses the induction with respect to $i$. To prove the base note that $2$-cages are boundaries
of $2$-simplices, and for each $2$-cage $k$ map $H_2$
provides the extension of the map of the boundary of the $2$-simplex to its interior. Moreover, the second property of $H_2$ implies that
if $k$ is in the closure of $M_m$, then its whole trajectory will be in $M^c_m$.

Now we prove the induction step. Assume that the proposition holds for all $i\leq I$. In order to prove it for $i=I+1$ consider
an $i$-cage $k$ and the one-parametric family of $i$-cages $k_t=H_i(k, t)$, $t\in [0,1]$. Each $i$-cage $k_t$ has $(i+1)$ of $(i-1)$-subcages
$k_t(l), l=1,\ldots ,i+1$ corresponding to the $(i-1)$-dimensional faces of the $i$-simplex $\sigma^i$. Applying the induction assumption
extend each $k_t(l)$ to the map $\phi(k_t(l))$ of the $(i-1)$-dimensional simplex $\sigma^{i-1}$. Together these $i+1$ maps provide
the extension of $k_t$ to {\it the boundary} of the $\sigma^i$. When $t$ varies in $[0,1)$, these extensions provide the extension
of $k$ to $\sigma^i$ minus a point $C$ at the center of $\sigma^i$, which is then mapped to $H_i(k,1)$. The continuity of the constructed extension at 
$C$ follows from property (3) of cage fillings and the continuity of $H_{i-1}$. This completes the proof of (1).

Now observe that for $j=i-1$ (2) follows immediately from the construction of $k$, 
and for $j=i-l$ immediately follows from this observation and the induction assumption.
Finally, (3) follows from property (2) of fillings of cages.
\end{proof}

\begin{Def} A (weak) filling of cages in $Cage^L_{n+1}$ is defined exactly as the strong filling with the only distinction that the requirement of continuity
of maps $H_i$ is replaced by a weaker requirement that only the compositions of $H_i$ with the quotient map $M^c\longrightarrow T=M^c/(\bigcup_i M^c_i)$ are continuous.
\end{Def}

\begin{Pro} \label{from cages to simplices}
Assume that $M$ admits a (weak) filling of cages for some $L\in (0,\infty]$ and $\delta>0$.
Let $k\in Cage^L_i$ be an $i$-cage in $M$ for some $i=2,3,\ldots ,n+1$. Then:

\begin{enumerate}
    \item There exists a map $\phi=\phi(k)$ of the $i$-simplex $\sigma^i$ to $M^c$ extending $k$. The composition $\tilde\phi$ of $\phi$ with the quotient map $M^c\longrightarrow T$ is continuous.
Further, the dependence of $\phi$ on $k$ becomes continuous
after projecting to $T$.
%More formally, $\phi(k)$ is a restriction of a map $\Phi:Cage^L_i\times \sigma^i\longrightarrow M^c$ to
%the points, where the value of the first variable is $k$; the composition
%of $\Phi$ with the quotient map $M^c\longrightarrow T$ is continuous.
\item Let $\sigma^j$ be a $j$-dimensional face of $\sigma^i$, ($2\leq j<i)$, $k_j$ the corresponding $j$-subcage of $\sigma^i$. Then $\phi(k_j)$ 
coincides with the restriction of $\phi(k)$ to $\sigma^j$.
\item Moreover, if the
image of $k$ (or one of its $j$-dimensional subcages $k_j$) is in a closed convex or $\delta$-locally convex subset of $M$ (in particular, it might be in the closure of $M_m$ for some $m=1,\ldots ,e$), the same will be true for the image
of $\phi$ (correspondingly, the restriction of $\phi$ to the $j$-dimensional face $\sigma^j$ of $\sigma^i$ corresponding to $k_j$.) 
\end{enumerate}
\end{Pro}

This proposition can be proven exactly as the previous one.

{\it Strong and weak fillings of flowers} are defined exactly as strong/weak fillings of cages. In fact, the definition of fillings of cages
implies that its restriction to flowers is a filling of flowers. In section 3.2 we will demonstrate that, vice versa, given a filling of flowers we can 
easily extend it to fillings of cages.

We would like to finish this section by observing that combining the previous proposition with the results in the previous section we obtain the following
proposition:

\begin{lemma} \label{no filling of cages}
Let $M$ be a complete non-compact $n$-dimensional manifold
with locally convex ends.
There exists $\delta_0>0$, $L_0>0$, such that for all $\delta \in (0, \delta_0]$ and $L\in [L_0, \infty)$ there is no weak filling of cages of length $\leq L$.
%If the assumption that $M$ have no nontrivial geodesic flowers implies that for some $\delta>0$ and all finite $L\in (0,\infty)$ there exists a weak filling of cages in $Cage^L_{n+1}$, then there exists a non-trivial geodesic flower on $M$. 
\end{lemma}

\begin{proof}
Consider the triangulation of $CS$ and a map $\psi$
from the $1$-skeleton of $CS$ to $M^c$ described in the previous section. 
If there exists a filling of cages, then by Proposition \ref{from cages to simplices} 
we can extend $\psi$ to a map defined on all of $C$, such that
the composition $q \circ \psi$ is continuous. Moreover,
the restriction of $\psi$ to $\overline{M}_0$ is homotopic
to the inclusion map $\phi$. 
By Lemma \ref{no extension} we obtain a contradiction.
\end{proof}

\subsection{From cages to flowers.}

Recall that some edges of a cage can be mapped by means of a constant map. Of course, in this case 
both endpoints are mapped to the same point. It can happen that a set of edges forming a spanning tree of a cage is being mapped
to the same point $p$ of $M$. In this case all vertices of the cage are mapped to $p$, and all other edges become loops based at $p$.
(Some of these loops can also be constant.) Recall, that we call such cages with one vertex {\it flowers}, and their non-constant loops {\it petals}.

Our next observation is that:
\begin{lemma} There exists  a deformation retraction  of the space of $i$-cages to the space of flowers with at most
${i(i-1)\over 2}$ petals with the following properties:
\par\noindent
(1) The image of a cage in $M$ does not change during this deformation. 
\par\noindent
(2) If the lengths of all edges
of a cage do not exceed $l$, then the lengths of all edges of the cage during the deformation (including the flower at the end of
the deformation) do not exceed $3l$. Therefore, the lengths of all cages during the deformation do not exceed ${i(i-1)\over 2}l$.
%\par\noindent
%(3) If $j<i$, then these deformation retractions for $i$ and $j$
%form a commutative square together with inclusions of the space of $j$-cages to the space of
%$i$-cages and the space of flowers with at most ${j(j-1)\over 2}$ petals to the space of flowers with at most ${i(i-1)\over 2}$ petals.
\end{lemma}

\par\noindent
{\bf Remark.} This deformation is unconditional, that is, it does not require any assumptions about $M$.

\begin{proof}
The idea is very simple. All vertices of a cage are numerated.
We just move all vertices but the one with the maximal number to the vertex with the maximal number, $v_{max}$, along the corresponding edge of the cage.
The speed is constant, and chosen so that the all vertices will collide with the maximal one at the moment $t=1$. All edges
between $v_i$ and $v_{max}$ shrink and become constant at $t=1$. On the other hand the segments $v_i(0)v_i(t)$ and $v_j(0)v_j(t)$ that
are being eliminated from $v_iv_{max}$ and $v_jv_{max}$, correspondingly, are being added to the edge $v_iv_j$
at both ends of $v_iv_j$. So, at the moment $t$
the edge between $v_i(t)$ and $v_j(t)$ will be the join of the three segments $v_i(t)v_i$, $v_iv_j$ and $v_jv_j(t)$.
\end{proof}

We make these deformation retractions the initial ``halves" of
$H_i$, and it will remain to ``fill" only the resulting flowers. (This means that given an $i$-cage we use the interval
$[0,{1\over 2}]$ of time to deform this cage to a flower with
the same image as in the proof of the previous lemma. We are going
to use the remaining time to "fill" the resulting flower.) So, we obtain the following corollary.

\begin{corollary}
Given a weak (correspondingly, strong) filling of flowers in $Flwr_j$, $j\leq {n(n+1)\over 2}$, there exists a weak (correspondingly strong) filling of $(n+1)$-cages. If $L$ is finite, then given a weak (correspondingly, strong)
filling of flowers $k$ in $Flwr^{nL}_j$, $j\leq {n(n+1)\over 2}$ via intermediate flowers $H(k,t)$ with length $\leq nL$, there exists a weak (correspondingly, strong) filling of $(n+1)$-cages of length $\leq L$ (via cages of length $\leq nL$).
\end{corollary}

{\bf Remark:} Our definition of filling of cages does not contain any restrictions for lengths of intermediate cages. Yet below we will be discussing a weak length non-increasing
filling of flowers of bounded length. The last assertion means that the existence of a weak filling of $(n+1)$-cages follows from the existence of a weak filling of flowers of length bounded by a larger value $nL$ via flowers of length bounded by the same constant. (The exact value of this larger value is not important for us, any $c(n,L)$ instead of $nL$
will work for us.) 

Combining  the previous corollary with Lemma \ref{no filling of cages} we see that:

\begin{lemma} \label{no filling of flowers}
Let $M$ be a complete non-compact $n$-dimensional manifold
with locally convex ends.
There exists $\delta_0>0$, $L_0>0$, such that for all $\delta \in (0, \delta_0]$ and $L \in [L_0, \infty)$ there does NOT exist a weak filling of flowers of length $\leq L$.
%If the assumption that $M$ have no nontrivial geodesic flowers implies that for some $\delta>0$, all finite $L\in (0,\infty)$ and all $j\leq {n(n+1)\over 2}$ there exists a weak filling of $j$-flowers of length $\leq L$ via flowers of length $\leq L$, then there exists a non-trivial geodesic flower on $M$.
\end{lemma}

\subsection{Flows on cages and fillings.}
Everything that we will say in this section about cages can be verbatim repeated for flowers. We will start from the following definition.

\forget
\begin{definition}
Given $\delta>0$ a curve-shortening flow on $n$-cages on $M$ is a family of continuous maps $F_i: Cage_i\times [0,\infty)\longrightarrow Cage_i$ for all $i\in\{2,\ldots, n\}$ such that for all $k, t, s>0$ $F_i(F_i(k,t),s)=F_i(k,t+s)$ and
that satisfies the following properties:
\par\noindent
(1) $F_i$ is length non-increasing, i.e. for each $k\in Cage_i$ and $t_1, t_2$ such that $t_1<t_2$, the length of $F_i(k,t_2)$ does
not exceed the length of $F_i(k,t_1)$. Moreover, if $k$ is not a stationary cage (in particular, not a constant cage), then the length of $F_i(k,t)$ is strictly less than
the length of $k$ for all positive $t$.
\par\noindent
(2) Let $k$ be a non-stationary $i$-cage. Then there exists an open neighbourhood $U$ of $k$ in the space of $i$-cages and positive $\lambda$ and $\epsilon$ such that for each $i$-cage $k_1\in U$ 
$%\vert
length(k_1)-length(F_i(k_1,\lambda))
%\vert 
\geq \epsilon.$
\par\noindent
(3) If an $i$-cage $k$ is a $j$-flower, then for all $t$ $F_i(k,t)$ is a $j$-flower.
\par\noindent
(4) If $k$ is stationary (possibly constant), then $F_i(k,t)=k$ for all $t$.
\par\noindent
(5) For each point $p$ there exists $r=r(p)>0$ such that if an $i$-cage $k$ is contained in the metric ball of radius $\rho\leq r$ centered at $p$,
then for some $t=t(\rho,p)>0$ $F_i(k,t)$ is a constant cage (= a point). Moreover, $\lim_{\rho\longrightarrow 0}t(\rho,p)=0$. 
% the suprema of $t(k)$ also tend to $0$.
\par\noindent
(6) If $C$ is a convex or $\delta$-locally convex subset of $M$, and an $i$-cage $k\subset  C$, then $F_i(k,t)\subset C$ for all $t.$
\end{definition}

%Below we will demonstrate the existence of a curve-shortening flow on a much larger space of {\it nets} for each positive $\delta$. Our exposition will essentially follow \cite{NR}
%with some minor improvements and (also minor) changes required to deal with the non-compact case. The restriction of this flow to cages or flowers 
%yields curve-shortening flows on cages or flowers.
Note that the only place where the chosen value of $\delta$ appears in this definition
is property (6). Essentially, we want the flow to preserve the ends $\bar M_i$, $i>1$,
under the condition of the main theorem.

Observe, that property (2) implies that for each compact set $W$ in $M$ there exists $t_W$ such that for all $t\geq t_W$ the image of $F_i(k,t)$ will be {\it outside} $W$ providing that the following two conditions hold: (1) There are no non-trivial stationary cages in the closure $\tilde{W}$ of the $length(k)$-neighbourhood of $W$; and (2) none of cages $F_i(k,t)$ are points (=constant cages). 
%then for each compact set $W$ in $M$ there exists $t_W$ such that for all $t\geq t_W$ the image of $F_i(k,t$ will be {\it outside} $W$.
Indeed, if this were not true, we would have a sequence of $i$-cages $F_i(k, t_n)$ of bounded length in the $length(k)$-neighbourhood of $W$ with
$t_n\longrightarrow\infty$. Using Ascoli-Arzela theorem we find a limit of a subsequence that we can denote $k_\infty$ in the closure of the $length(k)$-neighbourhood of $W$, and $length(k_\infty)$ is the limit of the decreasing sequence $length(F_i(k,t_n))$ as $n\longrightarrow \infty$.
Condition (2) implies that $k_\infty$ is stationary or a point. If it is a point, this point must be in $\tilde{W}$. Also, for all sufficiently large $n$
$F_j(k,t_n)$ will be sufficiently close to this point to ensure that $F_j(F_j(k,t_n),t)$ will become a point for a finite value of $t$ (property (5)). This conclusion contradicts our assumptions. Thus, we proved the following lemma:

\begin{lemma}
If $\delta>0$, and $F$ is a curve-shortening flow on $n$-cages on a complete non-compact Riemannian manifolds $M$ defined for this value of $\delta$, $i\leq n$, $k$ an $i$-cage, and $M$ does not contain a non-trivial stationary $i$-cage, then there is the following alternative:
\par\noindent
(1) $F_i(k,t)$ is a point for all $t\geq t_0$, where $t_0$ is some real number;
\par\noindent
(2) For each compact domain $W\subset M$ there exists a real $t_0$ such that $F_i(k,t)$ is outside of $W$ for all $t\geq t_0$.
\end{lemma}

In the second case we will say that $F_i$ drags $k$ to infinity. Exactly the same proof works in the case when $F$ is a curve-shortening flow on flowers.

Now we are going to prove the existence of the weak filling of cages (or, equivalently, flowers) assuming the existence of a curve-shortening flow
on cages (or even only on flowers).

\begin{lemma} Suppose there exists a weak curve-shortening flow $F$ on $(n+1)$-cages on a complete non-compact Riemannian manifold $M^n$.
%with $\delta$-locally convex ends. 
Assume that $M^n$ does
not contain a non-trivial stationary $i$-cage for all $i \leq n+1$.
Then there exists %an open neighbourhood $V$ of the convex core $M_0$ of $M$ and
a weak filling $H_i:Cage_i\times [0,1]\longrightarrow M^c$ for $i\in\{1,\ldots ,n+1\}$. 
%such that all $H_i$ are continuous on the set of all  $i$-cages with the image in $V$.
\end{lemma}

\begin{proof}
%Choose any positive $b$. Let $V$ be an open $b$-neighbourhood of $\bar M_0$.
Choose a positive integer $L$. We will construct $H_i$ for all $i$-cages of length $\leq L$.
Afterwards, we will somewhat modify this construction to make it continuous of $i$-cages of all length.

The basic idea if obvious: We will define $H_i(k,t)$ as $F_i(k,\lambda(k)t)$ for an appropriate $\lambda(k)$ that continuously depends on $k$ and $t\in [0,1)$. In other words,
we follow $F_i$ for time $\lambda(k)$ and then stop. $H_i(k,1)$ will coincide with $F_i(k,\lambda(k)t)$, if $F_i(k,\lambda(k))$ is a point (= a constant cage) in $M_0$ 
and will be the point at infinity of $M_i^c$, $i>0$ otherwise. %, if $F_i(k,\lambda(k))$ is in $\bar M_i$ but is not a constant cage.

The problem is how to determine $\lambda(k)$, so that this definition would make sense, and will work. Consider the domain $W$ that is defined
as the closed %$L$-neighborhood of $V$, or, equivalently $(b+
$L$-neighbourhood of $\bar M_0$. Recall, that for each $i$-cage $k$ $F_i$ either contracts it to a point in a finite time, or moves $k$ to infinity so that it eventually leaves forever any compact domain in $M$. 
\par
Consider the set $O_i(L)$ of $i$-nets $k$ of length $\leq L$ such that $F_i$ contracts $k$ to a point in (open set) $M_0$. If $F_i$ moves an $i$-net $k$ completely outside of $M_0$ at any moment of time, then $k$ cannot return to $M_0$ and is not in $O_i(L)$. Therefore, if $k\in O_i(L)$, then $k$ is contained in the interior of $W$, and, moreover, for each $t$ $F_i(k,t)$ is also in the interior of $W$. Consider the closure $\bar O_i(L)$ of $O_i(L)$.
If $k\in \bar O_i(L)\setminus O_i(L)$, it cannot be dragged to infinity by $F_i$ as this would force all sufficiently close $i$-nets in $O_i(L)$
to be dragged outside of $W$. By the same reason, $F_i(k,t)$ is contained in $W$ for all $t$. Similarly, $k\in \bar O_i(L)$ cannot be contracted by $F_i$ to a point outside of $\bar M_0$. Therefore,
all $i$-nets in $\bar O_i(L)$ are contracted by $F_i$ to points (=constant $i$-nets) in $\bar M_0$. For each $k\in \bar O_i(L)$ denote the minimal $t$ such that $F_i(k,t)$ is a constant net by $t(k)$. It is easy to see that $t(k)$ is upper-semicontinuous. Therefore, there exists $t_1(L)=\max_{k\in \bar O_i(L)}t(k)$. 

Consider the set $U_i(L)$ of $i$-nets $k$ of length $\leq L$ and contained in the interior of $W$ with the following property: There exists $t$ such that $F_i(k,t)$ is contained in the complement of $M_0$. Consider the infimum of such values of $t$, and denote it by $T(k)$. It is clear that for all $t\geq T(k)$ $F_i(k,t)$ will be outside of $M_0$ (that is, $F_i(k,t)$ will be in one of the closed sets $\bar M_j$, $j\geq 1$). We would like
to prove that $t_2(L)=\sup_{k\in U_i(L)}T(k)$ is finite. We are going to prove this
by contradiction. Assume the existence of a sequence $k_m\in U_i(L)$
with unbounded sequence of $T(k_m)$. Consider a convergent subsequence $\{k_{m_l}\}$ and denote its limit by $k_\infty$. If $F_i$ drags $k_\infty$ to infinity,
or to a point outside of $\bar M_0$, we see that at some finite moment of time $t$ $F_i(k_{m_l},t)$ will be outside of $\bar M_0$ for all sufficiently large values of $l$, 
and we obtain a contradiction. Similarly, it is easy to see that $F_i$ cannot contract $k_\infty$ to a point in $M_0$, as in this
case $F_i(k_{m_l}, t)$ will be in $M_0$
for all sufficiently large values of $t$ and $l$. If $F_i$ contracts $k_\infty$ to a point $p$ in $\bar M_0\setminus M_0$,
then $F_i(k_{m_l},t)$
will be in the ball of radius $r$ centered at $p$ mentioned in property (5) of the definition of curve-shortening flows for all sufficiently large
values of $l$. Therefore, $T(k_{m_l})\leq t+t(r,p)$ for $t(r,p)$ defined in property (5) and all sufficiently large $l$, and we again obtain a
contradiction. 
%As a corollary of finiteness of $t_2(L)$, $U_i(L)$ is a closed set. Indeed, if $F_j(k_m, t(L))$ is in $\bar M_j$ for all $m$
%and some $j\geq 1$, and $j_m\longrightarrow k_*$, then $F_j(k_*,t(L))$ is in $\bar M_j$ and, therefore, outside of $M_0$.

Define $\lambda(k)$ as $t(L)=\max\{t_1(L), t_2(L)\}+1$. 
Thus, $\lambda(k)$ is a function of only the length of $k$, and can be majorized by the same value
for all nets $k$ of length $\leq L$. Observe, that if $k\not\in \bar O_i(L)$, then either $F_i$ drags $k$ to the infinity, or $F_i$ contracts $k$ to a point outside of $M_0$.  
%
%Observe, that if the length of $k$ does not exceed $L$, 
%and $F_j(k,t(L)-0.5)$ is not a constant net, then either $F_j$ drags $k$ to infinity,
%or contracts it to a point outside of $M_0$. In both cases $F_j(k_*,t)$ will be outside of $M_0$ for all $t\geq t(L)-0.5$ and $k_*$ sufficiently close to $k$. 
%if $k$ is in $W$, then $F_j(k,t)$ will be outside of $M_0$ for all $t\geq \lambda (k)-0.5$. 
%If $k$ is not in $W$, then $F_j(k,t)$ will be outside of $M_0$ for all values of $t$.

The discontinuity of $H_i(k,t)$ is a priori possible only at $t=1$, and for $k\not\in \bar O_i(L).$ Moreover, if $F_i(k,t(L)-0.5)$ is a
stationary net, then $H_i(k,t)$ is continuous at $(k,1)$. In particular, we have continuity of $H_j$ at $(k,1)$ for all $k\in \bar O_i(L)$.
If $k$ has length $\leq L$ but is not in $W$, or $k$ is in the interior $U^0_i(L)$ of $U_i(L)$ regarded as a subspace of the space of all $i$-nets of length $\leq L$ in $W$, then the composition of the quotient map that sends the complement of $M_0$ to a point
$q$ will map $F_i(U,(t(L),\infty))$ to $q$ for some open neighbourhood $U$ of $k$. Thus, if $F_i$ drags $k$ to infinity, or contracts
it to a point outside of $\bar M_0$, then $k\in U^0_i(L)$, and we have the continuity after composing with the quotient map. It remains to
consider only $i$-nets $k$ of length $\leq L$ in $W$ that are neither in $\bar O_i(L)$, nor in $U^0_i(L)$. We are going to prove by contradiction
that such $i$-nets do not exist. Indeed, if $k$ is such a net, then
$F_i$ contracts $k$ to a point in $\bar M_0\setminus M_0$, %in time $>t(L)-0.5$, 
yet there exist arbitrarily close $i$-nets $k_j$ 
of length $\leq L$ in $W$ such that $F_i(k_j,t)$ intersects $M_0$ for all $t$. This implies that $F_i$ contracts each $k_j$ in a finite time
to a point in $M_0$, and, therefore, $k\in \bar O_i(L)$ - a contradiction.

Note that if we would replace $t(L)$ by any larger value, the above discussion will still remain valid. Moreover, the composition with the quotient map will be the same. Therefore, we can choose any a continuous function $\Lambda(L)$ on $(0,\infty)$ such that 
$\Lambda(L)\geq t(L)$,%  \lambda(\lceil x\rceil)$, 
and define $\lambda(k)$ as $\Lambda($length$(k))$. 
We can use this $\lambda(k)$ in the definition of $H_i(k,t)$ instead of $t(L)$.

\end{proof}
\forgotten

In order to prove our main theorem, we would like to consider curve-shortening
flow on cages of length $\leq L$ for some $L$, yet we need a weaker notion. Assume that $M$ is a complete non-compact manifold with locally convex ends (as in Definition 1.1).
Recall that $M^c\setminus M$ is a finite collection of points. Namely, it contains
one point $q_i$ for each end $M_i$. Let $cage_i^L$ denote $Cage^L_i\setminus\cup_j \{q_i\}$, where $q_j$ are regarded as constant $i$-cages.
%In a new version the flow 
One can map some pairs $(k,t)$ outside of the open $L$-neighbourhood
of $\bar M_0$ to points in $M^c\setminus M$ (regarded as constant cages not in $cage^L_i$) subject to the following conditions:

\begin{definition}
Assume that $M$ is as in definition 1.1, and $L$ is finite. A weak $L$-curve-shortening flow on $n$-cages on $M$ is a family of  maps $F_i: cage^L_i\times [0,\infty)\longrightarrow Cage^L_i=cage^L_i\bigcup M^c\setminus M$ for all $i\in\{2,\ldots, n\}$ such that for all $k, t, s>0$ $F_i(F_i(k,t),s)=F_i(k,t+s)$ as long as the images of all $F_i$ in this formula are in $Cage^L_i$ and
that satisfies the following properties:
\par\noindent
(0.1) $F_i$ is continuous on $F_i^{-1}(cage^L_i)$;
\par\noindent
(0.2) Assume that for some $k\in cage^L_i$ and for all positive $t$ $F_i(k,t)=q_j$ for some $j$. Then $k$ is contained in $M_j$ and has empty intersection with the closed $L$-neighbourhood of $\bar M_0$. On the other hand, there exists $\epsilon_1>0$ such that for each $k\in cage^L_i$
%There exists a positive $\epsilon$ with the following property: For each constant cage $q_j=(\bar M_j)^c\setminus \bar M_j$, ($j>0$), consider the set $S_{j\epsilon}$ of all cages $k$ in all pairs $(k,t)$ with $t>0$ contained in the $\epsilon$-neighbourhood of $F_j^{-1}(q_j)$. Then, if $k\in S_{j\epsilon}$, then $k$ is contained in $M_j$, and $k$ does not intersect the closed $L$-neighbourhood of $\bar M_0$. Further, there exists $\epsilon_1>0$ such that
contained in $M_j$, $j>0$, such that $k$ does not intersect the closed
$(L+\epsilon_1)$-neighbourhood of $\bar M_0$, $F_i(k,t)=q_j$ for all $t>0$.
\par\noindent
(0.3) If $F_i(k,t)=q_j$, then for all $T>t$ $F_j(k,T)=q_j$.
\par\noindent
(1) $F_i$ is length non-increasing, i.e. for each $k\in cage_i$ and $t_1, t_2$ such that $t_1<t_2$, the length of $F_i(k,t_2)$ does
not exceed the length of $F_i(k,t_1)$. Moreover, if $k$ is not a stationary cage (in particular, not a constant cage), then the length of $F_i(k,t)$ is strictly less than
the length of $k$ for all positive $t$. 
\par\noindent
(2) Let $k$ be a non-stationary non-constant $i$-cage. Then there exists an open neighbourhood $U$ of $k$ in the space of $i$-cages and positive $\lambda$ and $\epsilon$ such that for each $i$-cage $k_1\in U$ 
$%\vert
length(k_1)-length(F_i(k_1,\lambda))
%\vert 
\geq \epsilon.$
\par\noindent
(3) If an $i$-cage $k$ is a $j$-flower, then for all $t$ $F_i(k,t)$ is a $j$-flower (possibly $q_j$ regarded as a $j$-flower).
\par\noindent
(4) If $k$ is stationary (possibly a point), then $F_i(k,t)=k$ for all $t$.
\par\noindent
(5) For each point $p$ different from points $q_j$ there exists $r=r(p)>0$ such that if an $i$-cage $k$ is contained in the metric ball of radius $\rho\leq r$ centered at $p$,
then for some $t=t(\rho,p)>0$ $F_i(k,t)$ is a constant cage (= a point). Moreover, this point is not one of the points $q_j$, and $\lim_{\rho\longrightarrow 0}t(\rho,p)=0$. 
% the suprema of $t(k)$ also tend to $0$.
\par\noindent
(6) If $C$ is either $\bar M_j$ for $j\geq 1$ or a convex metric disc centered at 
a point $x\in \bar M_0$ of radius $<conv(x)$, and $k\in C$, then $F_i(k,t)$ is in $C$
for all values of $t$. Here $conv(x)$ denotes the convexity radius of $M$ at $x$.
%a convex or $\delta$-locally convex subset of $M$, and an $i$-cage $k\subset  C$, then $F_i(k,t)\subset C$ for all $t.$
\end{definition}

\begin{lemma} Let $F$ be a curve-shortening flow on $(n+1)$-cages of length $\leq L$ on a complete non-compact Riemannian manifold $M^n$ with locally convex ends $M_i$, $i=1,2,\ldots$.
%with $\delta$-locally convex ends. 
Assume that $M^n$ does
not contain a non-trivial stationary $i$-cages for all $\leq n+1$.
Then there exists %an open neighbourhood $V$ of the convex core $M_0$ of $M$ and
a weak filling $H_i:Cage^L_i\times [0,1]\longrightarrow M^c$ for $i\in\{1,\ldots ,n+1\}$. 
%such that all $H_i$ are continuous on the set of all  $i$-cages with the image in $V$.
\end{lemma}

\begin{proof}

If $k=q_j$, the $H_i(k,t)=q_j$ for all $t$. Now, it is sufficient to consider only cages $k\in cage^L_i$.

\par\noindent
%The proof essentially repeats the proof of the previous lemma. We need to check that
%possible discontinuities and various changes in definition do not affect the proof.
%We start from an analogue of Lemma 3.6:

{\it Step 1. For each $i$-cage $k$ of length $\leq L$ there exists $t=t(k)$ such that $F_i(k,t)$
is a point (possibly, one of the points $q_j$, $j\geq 1$).}

Indeed, the alternative is that there exists an unbounded increasing sequence $t_m$ such
that $F_i(k, t_m)$ is not a point. Therefore, these cages of length $\leq L$ are in the closed $(L+\epsilon_1)$-neighbourhood of $\bar M_0$. Pass to the limit of an appropriate subsequence. This limit, $k_\infty$, cannot be a point, as in this case
for a sufficiently large $F_i(k,t_m)$ will be in a small neighbourhood of this point,
and the flow will contract $F_i(k,t_m)$ to a point different from $q_j$ in a finite time. If $k_\infty$ is a non-trivial closed curve,then property (2) implies that
the flow will simultaneously decrease the lengths of curve $F_i(k, t_m)$ for all sufficiently large $m$ in time
$\lambda$ by at least $\epsilon>0$. Yet lengths of $F_i(k, t)$ decrease to the infimum equal to the length
of $k_\infty$, and we obtain a contradiction.

{\it Step 2. For each $i$ here exists $T=T(i)$ such that for all $i$-cages $k$ of length $\leq L$ either
$F(k,T)$ is a point, or its image does not intersect $\bar M_0$.}

Assume that there exists an infinite sequence of $i$-cages $k_m$ of
length $\leq L$
and an unbounded increasing sequence of times $t_m$ so that $F_i(k_m,t_m)$ is
not a point. This implies that all $k_m$ are in the (compact) closed $(L+\epsilon_1)$-neighbourhood of $\bar M_0$.
Passing to a subsequence, if necessary, we can assume that the sequence $k_m$
converges to an $i$-cage $k_*$. If $k_*$ is a point (that cannot be one of points $q_j$), then all sufficiently
large values of $m$ $k_m$ will be in a fixed small neighbourhood of this point.
Now property (5) of the definition above implies that the flow contracts all of them
to points in a uniformly bounded time, and we obtain a contradiction. If $k_*$ is not a point, but the flow contracts $k_*$ to a point $p$ different from all $q_j$, then
at some moment of time $t$ and all sufficiently large $m$ $F_i(k_m,t)$ will be
in a fixed small neighbourhood of $p$, and property (5) in the definition again yields
the contradiction. 
It remains to consider the case when the flow contracts $k_*$ to one of the points $q_j$. In a finite (possibly zero) time the flow moves $k_*$ out of the closed $L$-neighbourhood of $\bar M$, yet we choose this moment of time $t_*$ so that $F_(k_*,t_*)$ is not yet $q_j$. Therefore, for some $m_0$ and all $m>m_0$ $F(k_m,t_*)$
are also not in the closed $L$-neighbourhood of $\bar M_0$. Therefore, none of
$F_i(k_m,t_*)$ intersect $\bar M_0$. For $l=1,\ldots , m_0$ choose $t_l=t(k_l)$ so that $F_i(k_l,t_l)$ is a point. Now define $T$ as $\max\{t_*, t_1, \ldots, t_{m_0}\}$.

{\it Step 3.} Now we can define the weak fillings as follows. For each $i$
define $H_i(k,t)$ as $F_i(k, T(i)t)$ for all $t<1$. If $F_i(k,T(i))$ is a point different from all points $q_j$, we define $H_i(k,1)$ as the constant
cage $p$. Finally, if $H_i(k,T(i))$ does not intersect $\bar M_0$, it is contained
in some $M_j$ for some $j=j(k)>1$. In this case we define $F_i(k,1)$ as $q_j$.

\end{proof}

In the next section we are going to prove that:

\begin{theorem} Let $M^n$ be a complete $n$-dimensional Riemannian manifold with $\delta$-locally convex ends for some positive $\delta$. Assume that there is no
non-trivial geodesic flower on $M^n$. Then there for each finite $L$ and each $N$ there exists
a weak curve-shortening flow on $N$-cages of length $\leq L$ on $M^n$.

\end{theorem}

Combining this theorem with Lemma 3.4 and the flower version of the previous lemma,
we obtain Theorem 1.2. Thus, it remains only to prove Theorem 3.7. We are going to do this in the next section.

\forget
Observe that:
\par\noindent
(1) The set of $i$-cages that are contracted to a point is open. Correspondingly, the set
of $i$-cages that are being moved to infinity is closed.
\par\noindent
(2) The set of $i$-cages that are contracted by $F_i$ to a point in (open set) $M_0$ %the interior of $W$ 
is open.
Correspondingly, the set of $i$-cages that are either dragged to infinity of contracted
to a point outside $M_0$ %the interior of $W$ 
is closed. The intersection of this set with the
set of $i$-cages of length $\leq L$ and contained in $W$ is compact. Denote this set by $C_1(L)$.
\par\noindent
(3) The set of $i$-cages that are contracted by $F_i$ to a point in $\bar M_0$ %$W$ 
is closed.
Denote the set of such $i$-cages that are, in addition, contained in $W$ and have length $\leq L$ by $C_2(L)$. Note that the Ascoli-Arzela theorem implies that $C_2(L)$ is compact.
\par\noindent For each $i$-cage $k$ that is being contracted to a point by $F_i$ let
$t(k)$ denotes the time of the contraction to a point, i.e. the minimal $t$ such that
$F_i(k,t)$ is a constant $i$-cage. Note that $t(k)$ is an upper-semicontinuous function.
Therefore, it attains a maximum on $C_2(L)$. Denote this maximum by $t_2(L)$.
\par\noindent
(4) For all $i$-cages $k$ such that for all sufficiently large $t$ $F_i(k,t)$ is outside of %$M_0$. %$W$ 
$\bar M_0$,
%the interior of $W$, 
let $T(k)$ be the infimum of times $t$ such that %for all $t\geq T$
$F_i(k,t)$ is outside of %$M_0$. 
$\bar M_0$. 
%$W$. 
Note, the $T(k)$ is also an upper semi-continuous function. %on $C_1(L)$.
(Indeed, if $F_i(k,t)$ is in the (open) complement to %$W$
$\bar M_0$, then the same will be true for all sufficiently close to $k$ $i$-cages.)
Further, this set contains compact set $C_1(L)$.
Denote the maximum of $T(k)$ on $C_1(L)$ by $t_1(L)$.

\par
Now we can define $\lambda(L)$ as $\max\{t_1(L), t_2(L)\}+1$. %\max\{\max_{k\in C_1(L)}t(k), \max_{k\in C_2(L)}T(k)\}+1$ 
Now define $H_i(k,t)$ as $F_i(k,\lambda(L) t)$, if $t<1$;
$F_i(k,\lambda(L))$, if $t=1$,
and  $F_i(k,\lambda(L))$ is a point; and the point at infinity of $M_i^c$, $i>1$, if $t=1$, and $F_i(k,\lambda(L))$ is not a point.
%and is contained in $\bar M_i$. 

If $F_i(k,\lambda(L))$ is not a point, then it is not contained in $W$. This is the key point. Indeed, if (the image of) $k$ has empty intersection with $M_0$, then it is contained in the closure of one of $M_i$, $i\geq 1$. Property (6) of a curve-shortening flow implies that $F_i(k,t)$
will stay in $\bar M_i$ (and outside of $M_0$) for all values of $t$. If $k$ intersects $M_0$ and has length $\leq L$ then it is contained in $W$. Now $F_j(k,t)$ either eventually contacts $k$ to a point in $W$, or leaves the interior of $W$ forever (and never again intersects $M_0$).
At the moment of time $\lambda(L)$ one of these two events will have already happened.
%\forgotten

As the length of $F_i(k,\lambda(L))$ does not exceed $L$,
in the last case
$F_i(k,\lambda(L))$ is outside of $M_0$. The discontinuity is possible only for $t=1$ and only in the last case. Clearly, the discontinuity disappears after we take the composition with the quotient map collapsing the part of $M$ outside of $M_0$ to a point.
\forgotten

\forget
Recall that a triangulation of a subset of $M$ is fine if for each simplex $\sigma^i$ of the triangulation and each point $x\in \sigma^i$
$\sigma^i$ is contained in a metric ball $B(x,r)$ centered at $x$ of radius $r$ less than the convexity radius of $M$ at $x$, $conv(x).$
We will call such simplices {\it small}. When we fill the $1$-skeleton, $k$, of a small $1$-simplex in $\bar M_0$ as above,
at no time it can exit the convex ball $B(x,r)$. Therefore, the flow $F_i$ will contract $k$ to a point in $B(x,r)$, and $H_i(k,*)$ will be continuous on $[0,t]$. Further, $H_i$ will be a continuous function of both arguments on the product of a small neighbourhood of $k$ and $[0,1]$, and the map $\phi(k)$ constructed in Proposition 3.4 and filling $k$ with an $i$-simplex will be continuous.
Recall, that pairs of points in $B(x,r)$ can be connected by unique minimizing geodesics that continuously depend on the endpoints.
Therefore, there is an obvious homotopy between the identity map of $\sigma^i$, and the map $\phi(k)$ (or, more precisely, the composition of the map of $\sigma^i$  to the standard simplex and then $\phi(k)$). This homotopy will move the image of each point along the unique minimizing geodesic connecting it with the corresponding point.

This provides the missing details in the proof of main theorem by contradiction at the end of the previous section modulo the construction of a curve-shortening flow for flowers.
This flow will be constructed in the next section.
\forgotten
%Now we can return the the brief outline of the proof of the main theorem explained in
%the end of section 2. Consider a fine triangulation of $\Sigma_i$ for all $i$,
%extend it to a finite triangulation of $\bar M_0$. Now one can extend it to a triangulation $T_1$ of $S$ by cones over triangulations of $\Sigma_i$. Alternatively, one
%can consider the $1$-skeleton of this triangulation of $\bar M_0$, fill $1$-skeleta
%of all $n$-simplices of the triangulation by the maps $\phi(k)$

%\subsection{Proof of main theorem assuming the existence a curve-shortening flow for flowers.}

\forget
\subsection{ Adaptation of the idea of filling of cages to our situation.}

\par\noindent
{\it 1.} Proposition 2.5 implies that the remaining part of the construction of a continuous map $\Psi$ from $W$ to $M^c$
in section 2.1 (and, thus, of our main theorem) would follow from the following assertion:
{\it 
Assume that $M$ is a complete non-compact $n$-dimensional Riemannian manifold with locally convex ends without non-constant geodesic nets.
Then $M$ admits a strong filling of cages.} The only missing part will be a numeration of all vertices
in all $i$-cages for all $i$.This can be achieved simply by fixing some enumeration
of the considered triangulation of $M$, and enumerating vertices in each $i$-cage in accordance with the order of vertices in the chosen triangulation.

%However, we do not know how to prove the above assertion, or even if it is true. 
However, the above assertion may not be true since there may exist
a non-contractible cage in $M$ that every length decreasing
deformation pushes further and further into one of the ends of $M$.
Also, we do not need $\Psi$ to be continuous. 
Recall, that only the composition of $\Psi$ with the quotient map from $M^c$ to $T=\bar M_o/(\bigcup \Sigma_i)$ obtained by identifying $\bigcup_i M^c_i$ to a point is required to be continuous. Below we are able to construct
the (weak) filling of cages that develops discontinuities away from $\bar M_0$. %with images near constant cages in $M^c\setminus M$.
As the result, $\Psi:W\longrightarrow M^c$
will have discontinuities, but only at points in $\Psi^{-1}(M^c\setminus \bar M_0)$. Moreover, for each point $w\in W$ , where $\Psi$
is discontinuous, and each open neighbourhood $V$ of the point $m=q(M^c\setminus M_0)$ of $T$ there exists an open neighbourhood $U$ of $w$, such that $q(\Psi(U))\subset V\bigcup\{m\}$. As, the result the composition
of $\Psi$ and the quotient map $q$ becomes continuous at $w$.

We are going to be more specific about discontinuities of the cage filling that we are going to allow. Assume that the condition (1) in the definition of strong filling of $(n+1)$-cages is relaxed as follows:
Instead of demanding that $H_i(k,1)$ is a constant cage, we demand that either $H_i(k,1)$ is a constant cage in $M$, or $H_i(k,1)$ is a cage in one of the ends $M_j$ outside of some
open neighbourhood $V$ of $\bar M_0$.
Then we can define a discontinuous filling of cages $\tilde H_i$ as follows: First, rescale all $H_i$ so that the time interval would become $[0,{1\over 2}]$ rather than 
$[0,1]$. Then extend $\tilde H_i$ to the interval $[0,1)$ as a constant map on $[{1\over 2}, 1)$. Finally, define $\tilde H_i(k,1)$ as $H_i(k,1)$, if $H_i(k,1)$ is a constant cage in $V$, 
and as $c_l$, if $H_i(k,1)$ is in 
%$\bar M_j$. 
$(M \setminus V) \cap M_l$. % YL: not sure about this!
Recall that $c_l$ denotes the point used for the one-point compactification of $\bar M_l$. We need to verify that the maps $\Phi$ defined using $\tilde H_i$ (instead of $H_i$) as in the proof of the Proposition 2.5 becomes continuous after we compose
it with the quotient map $q:M^c\longrightarrow T$.
% It is sufficient to verify that 1) $\phi$ is continuous at all points in $\phi^{-1}(V)$; and
%2) For each $w\in \phi^{-1}(M^c_i)$, where $\phi$ is discontinuous, there exists an open neighbourhood $U$ of $w$ such that $\phi(U)\in M^c_i$.

Indeed, there are no discontinuities in $t\in [0,{1/2}]$ part of the filling corresponding to $H_j$, as well as in the $t\in [{1/2}, 1)$ part. The
discontinuities that arise at $t=1$ are
due to fillings of the interior ``halves" of some simplices by constant maps to $c_i$, and we need to verify that this process does not create any ``unwanted" discontinuities
near various types of points $(k,w)$, where $k$ is a $j$-cage, and $w\in\sigma^j$.
Recall that the discontinuities are allowed to happen only at the points $w$ in $\phi^{-1}(M^c\setminus V)$. Now, note that if $H_j(k,1)$ is outside of $V$, then
(a) for all $j$-cages $k'$ sufficiently close to $k$, $H_j(k',1)$ is outside of $\bar M_0$; (b) For all $j'$-subcages of $k$, $j'<j$, $H_{j'}(k',t)\in M^c_l$ for some
$l\in 1,\dots,e$; (c) the same
will be true for all $j'$-subcages of all cages $k'$ sufficiently close to $k$. These two observations easily imply the continuity of $q\circ \Phi$.

\bigskip\noindent
\forget
{\it 2.} Note that all maps $H_i$ in the definitions of strong filling and filling of $(n+1)$-cages, in fact, depend on $n$. So, it would be more appropriate to denote
these maps as $H^{n+1}_i$ to emphasize the dependence on the dimension of the simplices generating the continuous class of cages. Further, note that
we are going to apply the previous proposition, or rather its strengthening discussed in the previous subsection, to construct $\Psi:W\longrightarrow M^c$ filling
certain $(n+1)$-cages.
Therefore, whenever all maps
$\Phi:Cage_i\times \sigma^i\longrightarrow M^c$ need to be continuous (in the usual sense or a weak sense discussed above in subsection 2.2.A.1), when $i\leq n$, there is no need for any kind 
of continuity, when $i=n+1$. When $i=n+1$, it is just sufficient to define $\phi(k)$ on the (image to $M$ of the) $1$-skeleton $k$ of each considered (singular)
$(k+1)$-dimensional simplex relevant to the continuation to $W$.
Finally, we will need  to fill either $(n+1)$-cages that are either in $\bar M_0$ or in $\bar M_i$ for some $i\geq 0$.
For $(n+1)$-cage that are in $\bar M_0$ we have an upper bound for the length of each edge, namely, twice the diameter of $\bar M_0$ in the inner metric.
All $(n+1)$-cages in $\bar M_i$ can be just mapped by $\Psi$ to the point $c_i$. Of course, these creates the discontinuity as some of the relevant $(n+1)$-simplices
share a face (of dimensions up to $n$) with an $(n+1)$-cage with the image in $\bar M_0$ that has a non-empty intersection with $M_i$. Yet the property 2 of the definition of (strong)
filling of cages implies that the filling of any $j$-cage (or subcage) outside in $M_0$ will be outside of $M_0$ and, therefore, will be mapped to the same point $m\in T$ as $c_i$.
So, these discontinuities will disappear after taking the composition with the quotient map $q:M^c\longrightarrow T$.
Summarizing, it is sufficient for us to learn how to define fillings of $(n+1)$-cages with the image in $\bar M_0$ with edges of bounded length. Note, however,
that the construction of $\Phi$ involves fillings subcages of $(n+1)$-cages $H_{n+1}(k,t)$. Even when we start from $k$ in $\bar M_0$ and edges of controlled length,
$H_{n+1}(k,t)$ need not be in $\bar M_0$, and, in general, do not satisfy the same bound for the length of edges. Yet they will be at a bounded distance from $\bar M_0$ and will satisfy
{\it some} upper bound for the length of edges. This motivates the following definition:
\par

\begin{Def} For $r\geq 0$, $l$ we call $k$-cages $(k,l,r)$-small, if they are contained in the closed $r$-neighborhood of $\bar M_0$, and
the length of each edge does not exceed $l$.
\end{Def}

Observe, that for each triple $k,l,r$ the sets of all $(k,l,r)$-cages where each edge is parametrized proportionally to the arclength is compact.

Now we formally define {\it filling of $(k,l,r)$-cages} as follows:

\begin{Def} Let $N_i$ denote the set of all $(k,l,r)$-small $i$-cages. For each $j=1,\ldots, k-1$ let
$H^k_{k-j+1}:N_j\times [0,1]\longrightarrow Cage_{k-j+1}$ be a continuous map such that 
\par\noindent
(a) $H^k_{k-j+1}(*,0)$ are identity maps;
\par\noindent
(b) For each $c\in N_j$ $H^k_{k-j+1}(c, 1)$ is either a constant cage or a cage outside $1$-neighborhood of $\bar M_0$;
\par\noindent
(c) $N_{j+1}=\bigcup_{t\in [0,1]} H^k_{k-j+1}(N_j,t).$
\par\noindent
(d) If for some $c,t$ $H^k_{k-j+1}(c,t)$ is outside of $M_0$, then for each $t'>t$, $H^k_{k-j+1}(c,t')$ will be outside of $M_0$.
\par\noindent
This collection of maps is called a strong filling of $(k,r,l)$-small cages. If, instead of requiring that $H_{k-i+1}^k$ are continuous we require only that the composition of $H_{k-i+1}^k$ with the map of $Cage_{jk-i+1}$ induced by the quotient map from $\bar M_0$ to $T=\bar M_0/\cup\Sigma_i$ is continuous, then this
collection of maps is called a {\it (weak) filling of $(k,l,r)$-small cages}.

\end{Def}

\bigskip\noindent
\forgotten
\forget
\subsection{From cages to flowers.} 

Note that some edges of a cage can be mapped by means of a constant map. Of course, in this case 
both endpoints are mapped to the same point. It can happen that a set of edges forming a spanning tree of a cage is being mapped
to the same point $p$ of $M$. In this case all vertices of the cage are mapped to $p$, and all other edges become loops based at $p$.
(Some of this loops can also be constant.) Recall, that we call such cages with one vertex {\it flowers}, and their non-constant loops {\it petals}.

Our next observation is that:
\begin{lemma} There exists  a deformation retraction  of the space of $i$-cages to the space of flowers with at most
${i(i-1)\over 2}$ petals with the following properties:
\par\noindent
(1) The image of a cage in $M$ does not change during this deformation. 
\par\noindent
(2) If the length of all edges
of a cage do not exceed $l$, then the lengths of all edges of the cage during the deformation (including the flower at the end of
the deformation) do not exceed $3l$. 
%\par\noindent
%(3) If $j<i$, then these deformation retractions for $i$ and $j$
%form a commutative square together withe inclusions of the space of $j$-cages to the space of
%$i$-cages and the space of flowers with at most ${j(j-1)\over 2}$ petals to the space of flowers with at most ${i(i-1)\over 2}$ petals.
\end{lemma}

\par\noindent
{\bf Remark.} This deformation is unconditional, that is, it does not require any assumptions about $M$.

\begin{proof}
The idea is very simple. All vertices of a cage are numerated.
We just move all vertices but the one with the maximal number to the vertex with the maximal number, $v_{max}$ along the corresponding edge of the cage.
The speed is constant, and chosen so that the all vertices will collide with the maximal one at the moment $t=1$. All edges
between $v_i$ and $v_{max}$ shrink and become constant at $t=1$. On the other hand the segments $v_i(0)v_i(t)$ and $v_j(0)v_j(t)$ that
are being eliminated from $v_iv_{max}$ and $v_jv_{max}$, correspondingly, are being added to the edge $v_iv_j$
at both ends of $v_iv_j$. So, at the moment $t$
the edge between $v_i(t)$ and $v_j(t)$ will be the join of the three segments $v_i(t)v_i$, $v_iv_j$ and $v_jv_j(t)$.
\end{proof}

We make these deformation retractions the initial ``halves" of
$H^k_{k-j+1}$, and it will remain to ``fill" only the resulting flowers. (This means that given a cage from $N_i$ we use the interval
$[0,{1\over 2}]$ of time to deform this cage to a flower with
the same image as in the proof of the previous lemma. We are going
to use the remaining time to ``fill" the resulting flower.)
Let $Flwr_k$ denote the subset of $Cage_k$ formed by all flowers. (In other words, we consider only cages such that their restriction on a spanning tree of the
complete graph with $(k+1)$ vertices is constant.) Similarly, we call elements
of $Flwr_k$ $(k,l, r)$-small, if they are $(k,l,r)$-small as cages.
We call flowers
{\it $N$-flowers (or $N-(k,l,r)$-small flowers)},
if there are at most
$N$ non-constant loops (=petals). Recall, that for each element of $Flwr_k$,
$N\leq {k(k-1)\over 2}$.

The definition of filling of $(k,l,r)$-small cages can be easily amended
to become a definition of filling of $(k,l,r)$-small flowers:
We replace all occurrences of ``cage" or ``cages" by ``flower" or, correspondingly, ``flowers", and replace $Cage_{k-j+1}$ by $Flwr_{k-j+1}$.

Thus, the main lemma in the previous section follows from its analog for flowers:

\begin{lemma} If a complete Riemannian manifold $M^n$ with locally convex ends has 
no geodesic flowers
then for all $r$ and $l$ it
admits a filling of $(n+1,l,r)$-small flowers.
\end{lemma}

An obvious idea is to observe that the number of petals of $(n+1,r,l)$-flowers
does not exceed ${k(k-1)\over 2}$, and to prove this lemma using the inductive
proof with respect to the number of petals. The base of the induction
involves flowers with just one petal, or equivalently, closed
cures and says that: 

\begin{lemma} If a complete Riemannian manifold $M^n$ with locally convex ends has no geodesic flowers then for all $r$ and $l$ it
admits a filling of $1-(n+1,l,r)$-small flowers.
\end{lemma}

Substituting here the definitions of $1-(n+1,l,r)$ flowers and their fillings we see that the statement of the lemma can be  ``translated" as:

\begin{lemma} Let $M$ is a complete non-compact manifold with locally
convex ends and core $M_0$. Assume that there are no non-constant
geodesic flowers on $M$.
For each $l, r$ denote the set of closed curves of length $\leq l$ in the 
closure of the $r$-neighbourhood of $M_0$ by $\Lambda(l,r)$, and the space
of all closed curves on $M$ by $\Lambda M$.
Then there exists a continuous map $\alpha:\Lambda(l,r)\times [0,1]\longrightarrow \Lambda M$ such that for each $\gamma$ $\alpha(\gamma,0)=\gamma$
and $\alpha(\gamma, 1)$ is either a constant curve, or a closed curve outside
of the $1$-neighbourhood of $\bar M_0$.
\end{lemma}

This lemma will be proven in the next section, and its proof does not use any
novel ideas. We use the Birkhoff curve shortening process, and observe,
that if the application of this process to a curve $\gamma\in\Lambda(l,r)$ 
enters the closed $1$-neighborhood $(\bar M_0)_1$ for arbitrarily large 
values of time $t$ then all these closed curve cannot be too short (as 
otherwise the Birkhoff curve shortening flow will quickly shrink the short 
curve to a point). Then one can pass to a convergent subsequence that must 
converge to a non-constant periodic geodesic yielding a contradiction.
Therefore, we see that for each $\gamma$ there exists $t(\gamma)$
such that the Birkhoff curve shortening flow will also shrink $\gamma$
to a point in time $\leq t(\gamma)$ or for each time $T\geq t(\gamma)$
the application of the Birkhoff curve-shortening flow to $\gamma$
yields a closed curve outside of $(\bar M_0)_1$ at the moment $T$.
Now the compactness argument implies the existence of a uniform value
of $t=t(l,r)$ such that this $t$ has this property for all $\gamma\in \Lambda(l,r)$. Now one can define $\alpha$ as the result of the application of the Birkhoff
flow at the moment $t(l,r)$.

The induction step will be provided by the following lemma proven
in the last three sections of the paper. Its proof involves a combination of 
several old and new ideas:

\begin{lemma}Let $M^n$ be a complete non-compact Riemannian manifold with
locally convex ends and without non-constant geodesic flowers.
Denote the set of all $N-(n+1,l,r)$-small flowers by $Flwr_{n+1,N,l,r}$.
For each $l, r\geq 0, 2\leq N\leq {n(n+1)\over 2}$ there exist 
$L=L(n,l,r), R=R(n,l,r)$ and a continuous map
$\alpha^N:Flwr_{n+1,N,l,r}\times [0,1]\longrightarrow Flwr_{n+1, N, L, R}$
such that for each $\gamma\in Flwr_{n+1,N,l,r}$,
$\alpha^N(\gamma,0)=
\gamma$ and $\alpha^N(\gamma,1)\in Flwr_{n=1,N-1,L,R}$. (In other words,
$\alpha^N(\gamma,1)$ has at least one petal less than the original upper
bound $N$.)
\end{lemma}

Assuming this and the previous lemma we can immediately prove the existence of a filling
of $(n+1,l,r)$-small flowers by introducing $L-0=l, R_0=r, L_{i+1}=L(n,L_i,R_i), R_{i+1}=R(n,L_i,R_i)$ and simply combining homotopies $\alpha^{N-i}$ defined
on $Flwr_{n+1,N-i,L_i,R_i}$ for $i=0,\ldots, N-2$ as well as $\alpha$
from the previous lemma defined on $\Lambda(L_{N-1},R_{N-1})$.

Thus, the main theorem follows from the two lemmae above.
\forgotten

\medskip\noindent\section{Weak curve-shortening flow on spaces of flowers}
%Curve-shortening flow for closed curves and nets on complete non-compact Riemannian manifolds}

\forget
\subsection{Birkhoff flow for closed curves on non-compact manifolds} Observe that
$2$-cages are just closed curves (only separated into three arcs). For $2$-cages
one can use an adaptation of the Birkhoff curve shortening process (BCSP).
While this is not the flow that we are going to use for general nets (or even flowers), we decided to describe an adaptation of the Birkhoff curve shortening process for non-compact manifolds as 1) it is simpler than the flow that we use in this paper and will describe below; 2) The same (minor) technical difficulties that arise when we extend the Birkhoff process (flow) to complete non-compact manifolds arise also for our flow; 3) We are not aware of any descriptions of the Birkhoff flow in non-compact case.

%The main idea is to use the Birkhoff curve shortening process (BCSP) to contract $2$-cages that are just closed loops (separated into
%three segments).
Recall, that for closed Riemannian manifolds BCSP works as follows. The BCSP will be applied to
all closed curves of length $\leq L$ for some value $L$ of interest to us.

Initially, one chooses a large finite set
of points on the curve, so that the distances between consecutive points $p_i,\ p_{i+1}$ are equal to some parameter $v$ that is less than half of the injectivity radius  of the
manifold. The distance between the last of this points and the endpoint of  the curve does not exceed $v$ as well. For compact manifolds one can just take
$v$ to be equal to $\min\{\delta,{inj(M)\over 4}\}$. (Here $\delta$ is an arbitrary positive constant. We choose to be equal to $\delta$ in
Definitions 3.1, 3.2.)
 Then one replaces the curve by the piecewise geodesic curve  formed by all geodesic segments connecting consecutive
pairs of points $p_i$. One can connect the initial curve with the piecewise geodesic curve by means of a length non-increasing
homotopy. At each moment of this homotopy for each time $t\in (0,1)$ we first follow the arc of the initial curve from
$p_i=p_i(0)$ to $p_i(1-t)$, where $p_i(1)=p_{i+1}$, and then connect $p_i(1-t)$ with $p_{i+1}$.
We call this stage {\it Birkhoff deformation.}

Once the curve becomes piecewise geodesic the process becomes iterative. At each iteration one finds midpoints $\tilde p_i$ of geodesic segments
$p_ip_{i+1}$ and replaces the curve by a new broken geodesic curve formed by geodesic segments $\tilde p_i\tilde p_{i+1}$.
Again, there exists a continuous length non-increasing homotopy between the closed curve before and after this iteration.
For each $i$ and $t$ we follow the arc $\tilde p_i\tilde p_{i+1}$ parametrized by $t\in [0,1]$ for time $1-t$, and then cut
to $\tilde p_{i+1}$ along the minimizing geodesic.
At the end of each stage we do again the Birkhoff deformation with the same
parameter $v$. The purpose of the Birkhoff deformation is to rebalance the
distances between points $p_i$.
%If desired, after each stage one can move points $\tilde p_i$ along the closed curve
%to positions $\bar p_i$ such that all segments $\bar p_i\bar p_{i+1}$ have the same length and replace ala geodesic segments
%$\tilde p_i\tilde p_{i+1}$ by geodesic segments $\bar p_i\bar p_{i+1}$ (that, again, can be done by means of length non-increasing
%homotopy.

Observe that if the initial curve is in a locally convex domain $U$, then it will remain in $U$ forever, providing that the parameter $v$ is sufficiently small. This fact follows simply from
the fact that all new arcs on new closed curves are always minimizing geodesics connecting points on previously constructed
curves. If our initial manifold $M^n$ is closed, then
 1) we can use a uniform bound $N$ for the number of points $p_i$ that enables one
to treat BCSP as a flow on a finite dimensional manifold $i(M^n)^N$; 2) It is well known that a sequence of iterations of BCSP
applied to a closed curve $\gamma$ either converges to a point or to a non-trivial periodic geodesic; 3) If there is no
non-trivial periodic geodesic of length $\leq l$ (for some $l$), then BCSP is continuous
with respect to the initial curve.

If $M$ is a complete non-compact manifold without non-trivial periodic geodesics, there are two possible outcomes
of BCSP when applied to a closed smooth curve $\gamma$: Either this curve shrinks to a point, or it goes to infinity. The last
option means that for each $r$ there exists $N(r)$ such that after $N(r)$ iterations of BCSP the curve exists the
closed metric ball $B_r(v_0)$ of radius $r$ centered at $v_0$ and never returns to this metric ball again. Indeed,
proceeding by contradiction,
observe that the alternative is that the existence of some $r$ and an infinite increasing sequence $N_i$ such that
the curve after $N_i$ iterations of BCSP is in $B_r(v_0)$,
yet then it
exits $B_{2r}(v_0)$ (before again returning to $B_r(v_0))$.
This implies that the lengths of the results of all $N_i$th iterations of BCSP applied to the initial curve are bounded from below
by a positive constant since a curve of length $\leq c(r)=\inf$ convexity radius$(x)$ over $x\in B_r(v_0)$ will rapidly contract
to a point inside the ball of radius $c(r)$ centered at any of its points that is inside $B_{r+c(r)}(v_0)\subset B_{2r}(v_0)$.
On the other hand the lengths of these $N_i$th iterates do not exceed the length of the original curve. Now 
one can use the Ascoli-Arzela theorem to obtain a convergent subsequence of BCSP iterates, that converges
to a non-constant closed curve in $B_r(v_0)$ which is a fixed point for BCSP and, therefore, a periodic geodesic. This provides
a contradiction with the assumption that $M$ does not have periodic geodesics.

In both cases, BCSP provides a $2$-disc contracting a $2$-cage regarded as a closed curve (possibly to a point at infinity), and this process
is continuous with respect to the closed curve. %This completes one possible construction of $F_2$.

\medskip\noindent\subsection{Curve-shortening flow on spaces of nets.} 
\forgotten 

In this section we prove Theorem 3.7. We construct a weak curve shortening flow on the space of $L$-flowers.
All flowers outside of the
$(L+2\delta)$-neighbourhood of $\bar M$ will be immediately sent
to one of points $q_l$ at infinity in the same component of the complement of $\bar M_0$. Therefore, it is sufficient to consider $L$-flowers in the $(L+2\delta)$-neighbourhood of $\bar M_0$. Our description of this flow will be an adaptation of the process described in [NR] for cages on closed Riemannian manifolds, and does not contain any new ideas.

Let $i_1$ denotes the infimum of the injectivity radii of points of $M$ in this set, and $I$ denotes $\max\{\delta/2, i_1/4\}$. Let $N$ denote the integer part of ${L\over I}$. Given a flower
of length $\leq L$, we observe that the length of each petal does not exceed $L$.
Consider first the obvious length-nonincreasing deformation of the space of all $L$-flowers
in the $(L+\delta)$-neghbourhood of $\bar M$ into the space of broken
geodesic flowers with each petal subdivided into $N+1$ segments of equal length $\leq I$ by exactly $N$ intermediate points: One first subdivides the petal into $N+1$ arcs
of equal length using $N$ intermediate points. Then one connects pairs of consecutive points by (unique) minimizing geodesics. If the initial curve was outside of $M_0$, then the $\delta$-convexity of the complement of $M_0$ implies that the new curve will
be outside of $M_0$. Note that replacing each broken geodesic segment of the original
curve by the minimal geodesic segment in the new curve cannot increase the length.
We can connect the old curve and the new curve by a length non-increasing homotopy
by following the $N+1$ arcs of the original petal for shorter and shorter periods of time and then following the minimal geodesic to the end point of the arc.
This stage is similar to analogous stage in the Birkhoff curve shortening process.
Therefore, we will call it {\it Birkhoff deformation}.

Such flowers are completely determined by the coordinates of the base point of the flower and N $\times$ the number of petals of the flowers. 
From now on we can identify considered flowers with a subset of $(M^n)^K$,
where $K$ is $1+N\times $ the number of the petals, where "$1$" corresponds to the base point
of the flower. The flow will mostly consist of stages during which we flow $K$ points that determine the flowers along trajectory of a vector field on a domain in $(M^n)^K$.
Yet, we do not want distances between pairs of points that are supposed to be connected by the unique minimizing geodesic to grow too much.
It will be immediately clear from the construction of this field that the speed of movement of each point will not exceed $2(n+1)$. Therefore, we plan to stop after
each time interval of length ${I\over 10n}$. This guarantees that the distances between consecutive points will still be less that $\delta$, and less than $i_1\over 2$. Therefore, the corresponding points will uniquely determine a flower.
Right after stopping the flow we check if the curve is outside of $L$-neighbourhood of $\bar M_0$. If it is, it will be immediately mapped by the flow to the corresponding point at infinity $q_j$ and will stay there forever.
Then we perform the Birkhoff deformation, that will again drop the distances
between consecutive points on petals to less than $I$. At the end of the Birkhoff deformation we again check
if the resulting curve is outside the $L$-neighbourhood of $\bar M_0$, and map it to the corresponding point at infinity, if it is outside. 

Now we are going to describe the vector field on the considered domain of $(M^n)^K$.
We are going to describe it on several overlapping open subsets covering the considered domain, and then combine
these vector fields into one vector field by using a subordinate partition of unity.

We are first going to describe the flow on flowers where all petals are ``not too small", where ``not to small" means that the petal is not contained in ${a\over 2}$-neighbourhood f the base point of the flower for some small $a$ that will be
defined later.

The vector field depends on positions of all $K$ points that determine the flower.
For each of these points $p$ consider the adjacent geodesic segments on all petals that contain $p$.
The number of these segments will be equal to $2\times$ the number of petals
for the base point and $2$ for each of the remaining $K-1$ points.
For each of these geodesic segments consider the unit tangent vector at $p$ directed
from $p$ and sum up all these vectors. The result will be a vector $V=V(p)$ in the tangent space $T_pM^n$ of $M^n$ at $p$. These $K$ vectors will form the vector field at the point of $(M^n)^K$ that corresponds to the considered flower, and will be used to deform it.
The first variation formula for the length functional implies that the time derivative of
the length functional of the flower at $t=0$ will be equal to $-\Sigma_p\Vert V(p)\Vert^2$, where we sum over all $K$ points $p$
that determine the flower. Thus, the deformation will be length non-increasing,
and will be length decreasing unless all petals of the flower are geodesics, and the flower is a stationary geodesic flower. As we assumed that there are no stationary geodesic flowers,
the deformation will be uniformly length decreasing. Further, assume that the base point of the flower is on the boundary of $M_0$, yet all adjacent endpoints of geodesic segments are outside of $M_0$. Then all the tangent vectors at the base point point outside of $M_0$, and so is their sum. Thus, a flower outside of $M_0$
cannot, even partially, enter $M_0$ when deformed along this vector field.

Now consider all flowers where ALL petals are in a sufficiently small
neighbourhood of the base point $b$. In this case the base point does not move, and the remaining $K-1$ points move along the unique minimizing geodesics to the base point.
(In other world, the component of the vector field corresponding to the base point in the zero tangent vector, for all other points $p$ the components of the vector field
at $p$ are unit tangent vectors to the minimizing geodesics connecting $p$ and $b$ and
directed towards $b$. The size of the neighbourhood does not exceed $I$, and is chosen so that the flow is length-decreasing. It is easy to see that one can choose
the uniform size of such neighbourhoods over all points $b$ in the closed $(L+2\delta)$-neighbourhood of $\bar M_0$.

If some petals are very close to the base point, the remaining $l$ petals are not,
we move the flower as if it is a flower with just the $l$ "long" petals. We calculate
the vector field at the base point and all intermediate points on long petals as above (that is, as the sums of the unit tangent vectors to all incident geodesic segments). The tangent vectors to short
petals are not included into the sum at the base point. The components of the vector field at $N$ points $p$ on each short petal are then calculated as follows: First, consider the vector $V(b)$ calculated for the base point $b$ 
as the sum of unit tangent vectors to arcs of all long petals adjacent to $b$. (Each long
petal contributes two unit tangent vectors at both its endpoints that coincide with $b$). Consider the unique minimizing geodesic between $b$ and $p$ and parallel translate $V(b)$ to $p$ along this geodesic. Denote the result by $V_1$. Let $V_2$
denote the unit tangent vector at $p$ to the geodesic between $p$ and $b$. We direct this vector towards $b$. Define $V(p)$ as $V_1+V_2$. Looking at holonomies
for geodesic triangles $bp_ip_{i+1}$, where $p_i$, $p_{i+1}$ are adjacent
points on all short geodesic segments along "short" petals, we see that there exists $a>0$ such that
if all ``short" petals are in the $a$-neighbourhood of the base point (and the base point is in the $L$-neighbourhood of $\bar M_0$), then the first variation of the length
for the flow defined by this vector field will be negative.

Now we can consider a covering of the set of all consider $K$-tuples of points in $M^n$ by the open neighbourhoods of strata that correspond to degenerate flowers
where some or all petals are constant as well as an open set corresponding
to flowers where all petals are "long" (or, more precisely, not contained
in a small neighbourhood of the base point). Now we can use a subordinate partition
 of unity to combine the constructed vector fields on open sets into one vector field.
 The flow along this vector field will be length-decreasing.

\forget
Here we review the construction from [NR],
where a length-shortening flow was constructed on the space of nets on a closed Riemannian manifold, and extend it to complete non-compact manifolds.  Observe
that this flow has similar properties to BCSP; the generalization of the construction of this flow for complete non-compact manifolds is straightforward; and, its restriction to cages (or flowers) is a curve-shortening flow for cages (or flowers).
As we already proved in previous sections, the existence of such a flow implies our main theorem.
%here the conclusion of the theorem is replaced by a weaker conclusion that$M$ contains a geodesic net with at most $n+2$ vertices.

Recall
that a {\it net} is a piecewise-smooth map of a finite not necessarily connected
multigraph $G$ to a Riemannian manifold. This means that one first chooses some orientation
of all edges. Then one assigns to each directed edge $e=(v_i, v_j)$ a piecewise smooth map
$f_e:[0,1]\longrightarrow M$ subject to the following condition: If two edges $e_1$, $e_2$ are
both incident to a vertex $v\in G$, then the images under $f_{e_1},\ f_{e_2}$ 
of the endpoints of the intervals $[0,1]$ corresponding to $v$ must coincide. We regard $f_e$ as
the image of the edge of $G$ (parametrized by $[0,1]$) to $M$, and regard a net as a map 
$G\longrightarrow M$. 
Thus, $k$-cages are nets modelled on the complete graph with $(k+1)$-vertices.
A multigraph is allowed to have loops as well as multiple edges between two vertices. These multiple edges
can, in principle, be mapped in the same way creating a positive integer multiplicity (=weight) for the
curve in the manifold corresponding to this edge.
We call the  maps $f_e$ and sometimes their images $f_e([0,1])$ {\it edges of the net}, but
if it is important to distinuish between $f_e$ and its image, we call $f_e([0,1])$ {\it visible edges of the net}.
We call the images of the vertices of graph $G$ in $M$ {\it vertices of the net}.
%In principle, a vertex of the net can 
%correspond to a set of different vertices of the graph.
Further, the map $f_e$ is allowed to be constant on an edge $e$, so an edge of the net can look like a point.
Thus, the image of $G$ in the manifold can look like a quotient of the graph, where some of the ``visible" edges
may have integer weights $>1$.

A length of the net is, by definition, the sum of lengths of all its edges, or equivalently, the sum of lengths of all visible edges multiplied by their weights.
A net $Net$ is called a {\it (stationary) geodesic net} if it is a critical point of the length functional on the space of nets modelled on some multigraph $G$.
 More precisely,
this means that for each $1$-parametric flow $\Phi$ of diffeomorphisms of $M$ $0$ is the critical point of the
function $l(t)$ defined on the interval $(-\epsilon, \epsilon)$ for some $\epsilon>0$ by the formula $l(t)=$ length$(\Phi_t(Net))$. Using the
first variation formula for the length functional it is easy to see that a net is geodesic if and only if the following two conditions hold:
\par\noindent
1) Each of its edges is a geodesic.
\par\noindent
2) For each vertex $\nu_i\in G$ consider all unit vectors in the
tangent space $T_{Net(\nu_i)}M$ that are
tangent to the visible edges of the graph that are incident to $\nu_i$.
In each case we choose the direction
of the tangent vector from $Net(\nu_i)$ towards another endpoint regardless of the
orientation of the corresponding edge.
If an edge is constant, then we take the corresponding tangent vector to be the zero vector. 
The condition at $\nu_i$ is that the sum of all these tangent vectors multiplied by weights
of the corresponding visible edges is the zero vector in
$T_{Net(\nu_i)}M$.  Finally, the condition at $\nu_i$ must hold for all vertices $\nu_i$ of the graph $G$.

\par
Again, note that a stationary geodesic $k$-cage might look like a different graph in $M$.
For example, if a spanning tree of $G=K_{k+1}$ is mapped
into a point, and as the result all vertices merged (=were mapped) into one point, then
it will look as a collection of $K={k(k-1)\over 2}$ geodesic loops.
Further, some of these loops can be constant (and, thus, disappear), whereas some others can be mapped in the same way.
Thus, a (stationary) geodesic net {\it might} look like
a collection of weighted geodesic loops in $M$ with the same base point such that the sum of weights
is at most $K$. In addition, in this case the geodesic net satisfies the second condition, that asserts
%the second condition. In this case the second condition
that the (weighted) sum of all
unit tangent vectors at its only vertex (=the base point) directed from the vertex
is zero. A geodesic net that consists of
geodesic loops based at the same point is called a
{\it geodesic flower}, or {\it geodesic $m$-flower}, where $m$ is the number of non-trivial geodesic loops (``petals")
counted with their weights. Observe, that the stationarity condition at the base
point implies
that if there is only one (weighted) petal, then the geodesic net
must be an (iterated) periodic geodesic.

In [NR] one considers all nets in $M$ modelled on all graphs with the same number of edges $N$ as elements
of the same space. One can consider the distance between two nets as the minimum over all possible orderings
of edges of these nets of the maximal distance between edges of two nets with the same number.
The distance between two edges is the sum of the $C^0$-distance between them and $L^{\infty}$-distance
between their derivatives.
A length-shortening flow on the space of nets on a closed Riemannian manifold with a fixed
number (of possibly constant) edges can be defined as follows
%that deforms each net to a point or a non-trivial
%geodesic net 
(see the details in [NR]):

On the first stage one introduces finite collections of points on each edge dividing it into many
arcs of length less than, say, one quarter of the injectivity radius of the manifold. This is done using
the Birkhoff deformation applied separately to each edge:
One chooses a large finite set
of points on the curve, so that the distances between consecutive points $p_i,\ p_{i+1}$ are equal to some parameter $v$ that is less than half of the injectivity radius  of the
manifold. The distance between the second last of these points and the last one that coincides with the end of the edge does not exceed $v$ as well. The first of these points is the starting point of the edge. For closed manifolds (or compact domains on complete non-compact manifolds) one can just take
$v$ to be equal to $\min\{\delta,{inj(M)\over 4}\}$.
(Here $\delta$ is an arbitrary positive constant. We choose to be equal to $\delta$ in
Definitions 3.1, 3.2.)
 Then one replaces the edge by the piecewise geodesic curve with the same endpoints formed by all geodesic segments connecting consecutive
pairs of points $p_i$. One can connect the initial edge with the piecewise geodesic curve by means of a length non-increasing
homotopy. At each moment of this homotopy for each time $t\in (0,1)$ we first follow the arc of the initial curve from
$p_i=p_i(0)$ to $p_i(1-t)$, where $p_i(1)=p_{i+1}$, and then connect $p_i(1-t)$ with $p_{i+1}$.
We call this homotopy a {\it Birkhoff deformation.}
%Then one
%deformes each edge to the broken geodesic passing through all these new points and keeping
%the endpoints of the edge fixed exactly as it is done at the beginning of the BCSP.
A Birkhoff deformation 
constitutes a length non-increasing deformation of the space of all nets to the space of
piecewise geodesic nets. (Piecewise geodesic nets are, by definition, nets such that each edge is
a piecewise geodesic; no stationarity condition at vertices is assumed.)

If the lengths
of all edges are a priori bounded, and the manifold is closed,
then the total number of endpoints of geodesic intervals will be uniformly bounded. (The
same will be true if the manifold is not closed, but we consider only geodesic nets in a closed metric ball). Let us call all endpoints of the geodesic segments vertices. Some of these vertices
are original vertices of the geodesic net (these vertices are called ``multiple vertices" in [NR]), some are ``new" vertices added to subdivide
the edge into short arcs (and called ``double vertices" in [NR]). These vertices completely determine the geodesic nets. Therefore,
the considered space of geodesic nets with a bounded number of vertices is finite dimensional.
Yet this space is not a subset of full measure of
an appropriate Cartesian power of $M$. The reason is that
only certain pairs of vertices are supposed to be connected by geodesic arcs. Different choices
of connections correspond to different manifold strata, which intersect at lower-dimensional
strata corresponding to degenerate nets, where some vertices
merge and/or some
edges become constant
(=disappear). For example, consider piecewise geodesic nets with six vertices. These
six vertices can be possibly be connected into two disjoint geodesic triangles. When a vertex of
one of these triangles becomes close to a vertex of the other triangle and they eventually merge
one obtains a bow-tie shaped net with five vertices. However, the same bow-type shaped net
can be obtained from a geodesic hexagon connecting six vertices, where a pair of opposite vertices
becomes close and merge. Pairs of geodesic triangles and geodesic hexagons correspond to two different
strata of the space of nets with at most six vertices that can each be identified with a subset of
$M^6$. Their intersection in a neighborhood of the bow-type shaped net looks like a subset
of a degenerate stratum that can be identified with a subset of $M^5$.

From this moment the curve-shortening process will take place on the space of piecewise geodesic
nets (and will be quite different from the Birkhoff curve-shortening process).
Each net can be identified with a finite collection of points (vertices)
on $M$ some of which are supposed to be connected with geodesic segments
(=minimizing geodesics). The distance between each pair of such points
does not exceed half of the injectivity radius at either of them. Some of these
vertices (``old vertices") correspond to vertices of (multi)graph $G$ used to define the net, and some (``new vertices") were added to make distances between
consecutive pairs of vertices along each edge small.
We move a net by moving each of its vertices (that then can be connected by geodesic segments). Each vertex is moved along a trajectory of a vector field,
so we need to define these vector fields at all vertices. These fields
depend on the (whole) geodesic net but not on time.

A very natural idea is to define the velocity vector at each vertex $v_i$ as the
sum of all unit vectors tangent to edges incident to this vertex. (As usual,
we direct the tangent vectors away from the considered vertex.) Denote these sums $\Sigma_i$. The first variation
formula for the length functional implies that the derivative of length
in this direction is equal to the sum over all vertices of
$\Vert \Sigma_i\Vert^2$.

However, this choice of tangent vector fields leads to the following problem:
Consider the above example of two geodesic triangles converging to a geodesic
bow tie, and a geodesic hexagon converging to the same bow tie.
Consider the tangent vector fields for vertices of the bow tie, and, in particular, the tangent vector at the point where two geodesic triangles meet.
This tangent vector will not be close to tangent vectors at two vertices of
two triangles about to meet at the vertex of the bow tie. All these
tangent vectors are very different from tangent vectors at two vertices of the
geodesic hexagon about to meet at the vertex of the bow tie. Thus, the vector
field that we want to choose is discontinuous at the bow tie, or, more generally, at nets corresponding to lower-dimensional ``degenerate" strata.

On the other hand, the first variation formula for the length functional
implies that if one chooses the optimal vectors for vertices of the bow
tie as above and then uses parallel translation along minimizing geodesics
to move them to corresponding vertices of a nearby pair of disjoint
geodesic triangles or a nearby geodesic hexagon, then the derivative
of the length functional in the chosen direction will have almost the
same absolute value as the derivative of the length functional
for the bow tie. 

This suggests an idea of defining the desired flow inductively starting
from open neighborhoods of the most degenerate nets (namely, points or
finite discrete collection of points) and going up to higher and higher 
dimensional strata. This idea is implemented using partitions of unity
subordinate to a covering by open balls, where each ball is always centered at a point of the lowest dimensional strata that it intersects. The
tangent vectors for the net corresponding to the center of the ball
are sums of unit tangent vectors to the incident edges. The tangent
vectors at vertices of nets not at the center of the ball are obtained
by parallel translation of the tangent vector at the nearest vertex of the net
at the center of the ball. 
For example, assume that one of the petals of a flower became very short during the described process. As the result, the flower became close to a flower with a smaller number of petals that belongs to a stratum with a higher priority. Once the petal is sufficiently short
it will be ``dragged" after the base point (via parallel translations
of the velocity vector at the base point to vertices on the petal).
Moreover, the tangent vectors to the petal at the base point will
not be taken into consideration when we calculate the velocity vector
at the base point as if the petal disappeared.
%Nets contained in small neighbourhouds of points are contracted to their
%centers of masses  so thet the tangent vector at each vertex is always 
%tangent to the geodesic connecting the vertex and the center of mass.

In addition, one needs to move new vertices along ``old" edges to ensure that
distances between them remain small and almost equal.
For this purpose one frequently stops the flow applied to a geodesic net
and applies the Birkhoff deformation described in the previous section to every
edge. The goal is to ensure that the distance between double
vertices will not have time to increase by more than $inj/4$ (assuming
that $M$ is closed).
Note that if $M$ is closed, the intervals between two consecutive values
of stopping time can be uniformly bounded from below. If $M$ is not compact,
we define the same continuous function $i(r)$ as in the previous section and 
require that distance between consecutive points do not increase more than
by $i(r)$ during the interval between two consecutive values of the stopping time provided that the image of the net is all the time within the ball of radius
$r$ centered at $v_0$. In this way, the differences between consecutive
values of stopping time and the parameter $v$ of the Birkhoff deformation
will be uniformly bounded below by positive constants depending on $r$ while
the net is in $B_r(v_0)$. This implies that the number of double
vertices of nets will be uniformly bounded above by a function of $r$. Also, $v$ does not exceed the parameter $\delta$ characterising the local convexity of $\cup_i\Sigma_i$.
%Each new vertex is incident to
%two edges. Therefore, in order to move it in the desired direction
%we just add a positive multiple of the tangent vector to this edge
%(as always directed from the considered vertex) to the already constructed
%tangent vector. It is easy to check that this change does not
%decrease the speed of decrease of the length functional. In the non-compact case we might also need to keep adding or removing new vertices as the net
%moves towards infinity or back similarly to what we suugested in the previous section for BCSP.

Assume that a net is in a closed convex or $\delta$-convex
locally convex subset $C$ of $M$. In order to prove that
it will forever remain in $C$, it is sufficient to check only that for
each vertex of the net that is on the boundary of $C$ the tangent
vector at this vertex points towards $C$. This will  follow merely
from the fact that this tangent vector is a linear combination with positive
coefficients of tangent vectors to edges incident to the vertex all of which 
point to $C$. 

%Assume that $M$ admits no non-constant geodesic nets that are modelled on one of the
% that can be obtained from the complete graph with $n+2$ vertices
% a (possibly empty) sequence of operations of collapsing an edge.
%Then the constructed flow applied to a $i$-cage with $i\leq n+1$
%either shrinks it to a point, or moves it to infinity. The last option
%means that for each $r$ there exists time $t_0$ such that for all $t>t_0$
%the distance from $v_0$ to each point of the image of the net under the flow at
%time $t$ will be greater than $r$. The proof will be exactly the same 
%as the proof an analogous assertion for closed curves moving under the BCSP
%given in the previous section: If the net returns to a closed $r$-ball in $M$
%centered at $v_0$ for an infinite unbounded sequence of values of $t$, then
%passing to a convergent subsequence one obtains a geodesic net of non-zero 
%length, and, therefore a contradiction.

Observe, that when we apply the constructed flow to a cage, it remains a cage at all times. The same will be true for a flower.
It is also easy to check that the flow on cages (correspondingly, flowers) has all the required properties in Definition 3.3

Thus, we obtain the following theorem:

\begin{theorem} 
For each $N$ there exists a curve-shortening flow on $N$-cages. The restriction of this flow to flowers is a curve-shortening flow on flowers.
\end{theorem}

Thus, we arrive to the following proposition that follows from Proposition 3.4:

\begin{Lem}
Assume that there are no non-constant geodesic nets on $M$. The $M$ admits
a filling of cages.
\end{Lem}

In view of the previous discussion we see that:

\begin{Cor}
A complete non-compact manifold with locally convex ends has a non-constant 
geodesic net.
\end{Cor}
\begin{proof}

The main idea is obvious: Assuming that there are no geodesic nets we are attempting the impossible extension as described above.
We already explained how to extend to the $1$-skeleton. Now we use existence of the filling of cages to obtain the impossible extension that provides
the contradiction with Corollary 2.2.

However, this argument has a minor flow. Namely, when we need to extend $\Psi$ (defined in section 2) to an
$(n+1)$-dimensional simplex $\sigma^{n+1}$ , we need not only to extend the already constructed restriction of $\Psi$ to the $1$-skeleton of $\sigma^{n+1}$ but also to ensure that the image of the restriction of $\Psi$ to the face $\sigma^m$ of $\sigma^{n+1}$ in $M^n$ coincides with this face. Fortunately, we are helped by the fact that these faces in $M^n$ are all very small, and are in the ball of radius smaller than the convexity radius of the manifold. As the result, the $1$-skeleton of each small $m$-dimensional simplex in $M^n$ will be contracted to a point within the ball $B$ of radius $conv(x)$ centered at any vertex $x$ of this simplex.
The same will be true for all its subcages and subcages of the nets obtained during the curve-shortening flow. Thus, the $1$-skeleton of $\sigma^{n+1}$ will be filled inside $B$. While it will not coincide with $\Psi$ , it will be homotopic in $B$ to $\Psi$, as we can connect the images of respective points simply along the minimizing geodesics in $M^n$. (Here we using the fact that each pair of points within a metric ball of radius $conv(x)$ for any $x\in M^n$ can be connected by the unique minimizing geodesic, and this geodesic continuously depend on its endpoints.) The restrictions of these homotopies on small $n$-dimensional simplices of the triangulation of $M^n$ (regarded as $n$-faces of corresponding $(n+1)$-dimensional simplices) can be glued into one homotopy defined on all $M^n$. This follows from the fact that fillings of $(n-1)$-dimensional faces of $n$-simplices of $M^n$ do not depend on the ambient $n$-dimensional simplices but only
on the $(n-1)$-dimensional face itself (see Proposition 3.4 (2)).
\end{proof}

\par\noindent
{\it Proof of Theorem 1.2 from Lemma 3.5} It is almost straightforward to prove 
Theorem 1.2. Indeed, Lemma 3.1 reduces filling of cages to filling of flowers, and
a brief examination of our curve-shortening flow shows that when we restrict it to flowers, its trajectory goes via flowers. Also, the trajectory of a flowers in a convex ball in $M^n$ forever remains in this ball. This helps to resolve the complication that the restriction of $\Psi$ that we construct to $M^n$ is not the identity map. (Indeed, as before we see that this map is homotopic to $\Psi$, and the contradiction with Corollary 2.2 still arises.)

\par\noindent
{\it Proof of Lemma 3.5:} The constructed curve-shortening flow on cages restricts to a curve-shortening flow on flowers. In the absence of a stationary flower this flow can be continued indefinitely.

%\forget
While periodic geodesics, iterated periodic geodesics or ``unions"
of iterated periodic geodesics are all geodesic nets, there are many geodesic
nets that have nothing to do with periodic geodesics. Therefore, one would
need new ideas to prove the existence of a periodic geodesic.

\par\noindent
{\bf 2.5. Weighted length functional.} In this section we
are going to review ideas of the third author (R.R.) on applications of weighted length functionals that can be found in [R2], [R3], [R4]. The starting point
is that instead of the length functional one considers weighted length functional, where the lengths of edges are multiplied by powers of some number $a\geq3$
that later is chosen to be very large. This implies that in the stationarity condition each unit vector tangent to an edge is being multiplied by
the corresponding power of $a$. Similarly, factors equal to powers of $a$
enter the formula as coefficients at tangent vectors at all vertices.  When
an ``old" edge is broken into many short ``new" edges the weights for all
these new edges are equal to the weight for the old edge.

\par\noindent
{\bf 2.5.1. Idea 5:} Let $k$ be a cage with vertices
$v_,\ldots , v_m$. We enumerate edges of this cage as follows: The edge
connecting $v_2$ to $v_1$ has number $1$, edges connecting $v_3$ to $v_1$
and $v_2$ have numbers $2$ and $3$, etc. The last $m-1$ numbers are given to edges connecting $v_m$ with $v_1,\ldots , v_{m-1}$. Each edge is given the weight
equal to $a^i$, where, $a\geq 3$ is a chosen constant, and $i$ is the number
of the edge. The weighted length functional
is defined as $\Sigma_{i=1}^{m(m-1)/2} a^ilength (e_i)$. 
Note that is $a$ is integer, and $k$ is the critical point for this functional,
then the new net where the weight of each $e_i$ is multiplied by $a^i$ 
will be a critical point for the weighted length functional, i.e. a geodesic net.

The merging lemma from [R2] asserts that if $a\geq 3$, then each critical
point of the weighted length functional is a geodesic flower (see above for the
definition of the geodesic flowers). In particular, it consists of $N$
geodesic loops based at the same point, where $N\leq m(m-1)/2$.

Indeed, otherwise there exists an edge connecting two distinct vertices.
Among all such edges choose one with the largest number $i$. Consider
its endpoint that has a smaller number.  Consider the balancing condition
at this edge. The unit vector tangent to $e_i$ enters with weight $a^i$.
Each of the other edges incident to the same vertex contributes either a vector of length $a^j$ for $j<i$ (taken into account the weight), if it connects
different vertices, or 
two unit vectors with weight $a^j$,  $j<i$, if this edge looks as a loop.
The sum of all these vectors with the exception of the unit vector
tangent to $e_i$ has length that does not exceed
$2(1+a+a^2+\ldots +a^{i-1})<a^i$, and so they cannot balance $e_i$.

This would provide an alternative approach to deforming cages to flowers 
on a closed manifold without non-constant periodic geodesics (that was done in a more straightforward way above):

%Define a {\it flower} as a finite collection of loops with common
%endpoint. We do not require neither that the sum of all unit tangent vectors at the common basepoint is equal to
%$0$, nor that the loops are geodesic.
%Define spaces of $k$-flowers as
%subsets of spaces of $k$-cages that consist of cages
%that are constant on a spanning tree of the $1$-skeleton of the $k$-simplex.
%(So, these cages look like a flower. Note that a $k$-flower can have up to
%${k(k-1)\over 2}$ non-trivial loops.)
%
%We observe that the discussion in section 2.5.1
%easiy implies the following lemma:

\begin{Lem}
Assume that $M$ is a complete non-compact Riemannian manifold without non-trivial periodic geodesics. Then for each $k$ the space of $k$-cages can be deformed
to the space of $k$-flowers. 
\end{Lem}

\begin{Pf} The idea is to choose a sufficiently large $a$ (say, $a=4$),
to do the weighting of edges by powers of $a$ as in section 2.5.2, and to use
the weighted length decreasing flow defined exactly as in section 2.4, only for the weighted length functional instead of the length functional.
(Practically, this would mean that when one is calculating the velocity
vector at a vertex, one takes the sum of all unit tangent vectors multipliedby the weights of corresponding edges, instead of merely the (unweighted) sum of the tangent vectors.)
This process will be decreasing the length of the edge with the highest
weight at a rate uniformly bounded below. Once this edge collapses to a point
the same will be happening with the next edge, etc. After passage of time
that depends continuously on a cage it will be deformed to a flower.
\end{Pf}

{\bf 2.5.2. Idea 6:  Wide geodesic loops ([R3])} Assume that $M$ is
a closed manifold, and $c$ is a
non-trivial critical point of the weighted length functional defined as above using powers of $a\geq 3$ on the space of $i$-cages, where $i\leq n+1$.
Then as it was observed in the previous section it will be a geodesic flower
with $N\leq (n+1)n/2$ petals. Consider the petal that corresponds to the
edge with the largest number. Denote the angle at the base point of this petal
by $\alpha$. It had been observed in [R3] that as $a\longrightarrow\infty$
$\pi-\alpha\longrightarrow 0$. Indeed, the sum of two tangent vectors
at the base point for this geodesic loop has length
$2\sin({\pi-\alpha\over 2})$. Therefore, $2a\sin({\pi-\alpha\over 2})\leq
2(N-1)$, which implies $\pi-\alpha=O({1\over a})$ as $a\longrightarrow\infty$
as $N\leq (n+1)(n+2)/2$. Of course, this includes the possibility 
that $\alpha=\pi$, which would mean that the petal is a non-trivial periodic geodesic. 

Assume that $M$ is closed or complete non-compact. Then an easy calculation implies that
if the length of each edge of the initial net is $\leq L$ for some $L$,`
the net is modelled on $K_{n+2})$,  and a geodesic flower
was obtained as a limit point of the weighted length shortening flow
upplied to this net as above for some value $a\geq 3$
has the property that the edge with the maximal number survives as a
non-constant petal, then the length of this petal is at most
$(1+a+\ldots +a^N)L/a^N<L(1+{1\over a-1})<L(1+O({1\over a})<2L$, where
$N=(k+1)(k+2)/2$.

If $M$ is a complete, non-compact manifold that an easy compactness argument
implies that for each $r>0$ and $L$ if there are no non-trivial
periodic geodesics of length $\leq L$ in the ball of radius $r$ centered at $v_0$, then there exists $\alpha_0=\alpha_0(r,L)<\pi$ such that each geodesic
loop of length $\leq L$ in $B_r(v_0)$ has angle $\leq\alpha_0$. This implies
that:

\begin{Lem} Let $M$ be a complete Riemannian manifold that has no
periodic geodesics of length $\leq L$ in the closed ball centered
at point $v_0\in M$. Let $c$ be an $i$-cage, where $i\leq n+1$, and the
length of each edge of the cage is $\leq L$. There exists $a=a(L,r)>3$
with the following property: Apply the weighted length shortening
flow constructed as in section 2.4 to $c$, where weights are the powers of $a$ as described in section 2.5.1. Assume that a geodesic flower is in $B_r(v_0)$
and is the limit of an unbounded sequence of instances of the flow applied
to $c$ and contained in $B_r(v_0)$. Then the petal corresponding
to the edge with the largest number is trivial. In orther words, the limit of lengths of this edge for the considered values of time is zero.
\end{Lem}

Why cannot we assert the same for the petal corresponding to the edge with the
second largest number? The corresponding weight is $a^{N-1}$, and the
length of this petal in the stationary flower is $\leq {a^{n+1}-1\over a-1}L/a^{N-1}=aL(1+O({1\over a}))<2aL$. Therefore, we cannot assert for
$a=a(L,r)$ that in the absence of periodic geodesics a geodesic loop
of length $\leq 2a(L,r)L$ will have angle $\leq\alpha_0(L,r)$.

This lemma implies that if the petal corresponding to the edge with the largest number
does not shrink to a point, it at least can be made to move as far away as we want,
in particular, outside any chosen by us neighbourhood of $K$. This also implies that the base point
of all petals of the flower is well outside of $K$. Now one wants to move the second
petal (that might have become much longer than it was) also outside of $K$, etc.
(Recall, that if we will manage to move all petals well outside of $K$, then we are done.)
A natural idea might be to change the weights, so that the second petal becomes
the highest weighted. If we do that, then the previous lemma guarantees that the weighted length
shortening flow for the new weighting will, indeed, move the second petal otside of $K$.
Unfortunately, in the process the base point might have moved back to $K$ (and, then out of $K$),
and a part of the first petal might now return back to $K$. So, we do not see a way to proceed
and need to apply our assumption that $M$ is a covering space of the closed Riemannian manifold
with the pullback Riemannian metric.

If $M$ is the covering space of the closed Riemannian metric then each geodesic loop
can be shifted to a geodesic loop in a fixed fundamental domain. Therefore, the compactness
argument before the text of the previous lemma implies that if $M$ also has no
periodic geodesics, then for each $L$ there exists $\alpha_0=\alpha_0(L)<\pi$ such that each geodesic
loop in $M$ has angle less than $\alpha_0$.

Now we have the following strengthening of the previous lemma:

\begin{Lem} Let $M$ be a complete Riemannian manifold that has no
periodic geodesics of length $\leq L$. Assume that $M$ is the covering space of a closed Riemannian manifold.
Let $c$ be an $i$-cage, where $i\leq n+1$, and the
length of each edge of the cage is $\leq L$. There exists $a=a(L)>3$
with the following property: Apply the weighted length shortening
flow constructed as in section 2.4 to $c$, where weights are the powers of $a$ as described in section 2.5.1. Assume that a geodesic flower
is the limit of an unbounded sequence of instances of the flow applied
to $c$. Then the petal corresponding
to the edge with the largest number is trivial. In other words, the limit of lengths of this edge for the considered values of time is zero.
\end{Lem}

Now the lengths of the remaining edges of the flower might have become larger but they still
cannot exceed $L_1=a^{N+1}L$. Now we can apply the lemma again for a new value of $a$
defined for $L_1$ instead of $L$. The weighted length decreasing flow for this new system of weights
will converge to a flower where the two petals with largest weights shrank to points, etc.
Proceeding inductively, we will shrink the whole geodesic flower (or the initial cage) to
a point after a finite number of applications of the weight length decreasing flow
for an increasing sequence of values of $a$. To make this procedure continuous, the moments
of reweighting need to be chosen independent of the initial cage (or geodesic flower).
To do so, note that an obvious compactness argument implies the existence of time $T=T(L)$
such that the petal corresponding to the largest weight will shrink to a point for all
$i$-cages in the text of the previous lemma. This enables us to start from the value of $L$
of interest for us, and inductively define a finite sequence of values of time, when we perform 
the reweighting (and continue the flow for the weighted length functional corresponding to the new
system of weights).

\forgotten
\medskip\noindent
{\bf Acknowledgements.} This work was started
during the authors' visit to the Bernoulli Center of the EPFL in Summer 2017.
Another part of this work was done while two
of the authors (Alexander Nabutovsky and Regina Rotman) visited the Max Planck of Mathematics of Mathematics in June-July, 2021.
The authors would like to thank both the Bernoulli Center of the EPFL and
the Max Planck Institute of Mathematics for their
kind hospitality. 
 
 The research of Yevgeny Liokumovich was partially supported by NSF Grant DMS-1711053 and NSERC Discovery grant RGPAS-2019-00085.
 
 The research of Alexander Nabutovsky was
 partially supported by his NSERC Discovery Grant RGPIN-2017-06068. The research of Regina Rotman was partially supported by her NSERC Discovery Grant RGPIN-2018-04523.

 The research of Gregory R. Chambers was partially supported by NSF Grant DMS-1906543.
\medskip
%However, it will have angle $\leq\alpha_1=\alpha_0(2a(r,L),r)<\pi$.
%Therefore, there exists $a_1(L,r)=O({1\over\pi-\alpha_1})$ such that if

%\par\noindent
%\par\noindent

%\normalsize
\vskip 1truecm

\begin{tabbing}
\hspace*{7.5cm}\=\kill
Gregory R. Chambers        \>Yevgeny Liokumovich \\                         
Department of Mathematics, \>Department of Mathematics \\          
Rice University,           \>University of Toronto  \\             
MS-136, Box 1892           \>Toronto, Ontario M5S2E4 \\            
Houston, TX 77251, USA     \>Canada \\                          
gchambers@rice.edu         \>ylio@math.toronto.edu \\      
                        \>\\
                         \>\\

Alexander Nabutovsky                \>Regina Rotman\\
Department of Mathematics           \>Department of Mathematics\\
University of Toronto               \>University of Toronto\\
Toronto, Ontario M5S2E4             \>Toronto, Ontario M5S 2E4\\
Canada                              \>Canada\\
alex@math.toronto.edu               \>rina@math.toronto.edu\\

\end{tabbing}

\end{document}